\documentclass[11pt]{article}
\usepackage[margin=1in,a4paper]{geometry}
\usepackage{commath}
\usepackage{amsfonts}
\usepackage{amsthm}
\usepackage{amsmath}
\usepackage{amssymb}
\usepackage{enumerate} % to change the enumerate label
\usepackage[longnamesfirst]{natbib} % this is necessary to invoke ecconometrica bibtex style
\usepackage[bookmarks,colorlinks = true,linkcolor = blue,urlcolor  = blue,citecolor =
blue,anchorcolor = blue]{hyperref} % Create books for pdf file
\usepackage{booktabs} %To produce nice table
\usepackage{xr} % For cross-referencing
\usepackage{color} % To make text red
\newtheorem{lemma}{Lemma}[section]
\newtheorem{example}{Example}[section]

\numberwithin{equation}{section} % Number equation by section
\DeclareMathOperator*{\ve}{vec}
\DeclareMathOperator*{\vech}{vech}

\begin{document}

\title{Supplementary Material for "Estimation of a Multiplicative Correlation Structure in the Large Dimensional Case" %\thanks{We would like to thank Heather Battey, Michael I. Gil', Liudas Giraitis, Bowei Guo, Xumin He, Chen Huang, Anders Bredahl Kock, Alexei Onatskiy, Qi-Man Shao, Richard Smith, Chen Wang, Tengyao Wang, Tom Wansbeek, Jianbin Wu, Qiwei Yao for useful comments and/or help in one way or another.} 
}

\author{Christian M. Hafner\thanks{Institut de statistique, biostatistique et sciences
actuarielles, and CORE, Universit\'e catholique de Louvain, Louvain-la-Neuve,
Belgium. Email: \texttt{christian.hafner@uclouvain.be}.}\\Universit\'{e}{\small \ }catholique de Louvain
\and Oliver B. Linton\thanks{Faculty of Economics, Austin Robinson Building,
Sidgwick Avenue, Cambridge, CB3 9DD. Email: \texttt{obl20@cam.ac.uk}.}\\University of Cambridge
\and Haihan Tang\thanks{Corresponding author. Fanhai International School of
Finance and School of Economics, Fudan University, 220 Handan Road, Yangpu District, Shanghai, 200433, China. Email:
\texttt{hhtang@fudan.edu.cn}.}\\Fudan University}
\date{\today}
\maketitle

\section{Supplementary Material}

This section contains supplementary materials to the main article. SM \ref{sec B1} contains additional materials related to the Kronecker product (models). SM \ref{sec two stage MD} outlines a shrinkage approach via minimum distance to make the estimated $\exp(\log \Theta_j^0)$ indeed a correlation matrix for $j=1,\ldots,v$.  SM \ref{sec A5} gives a lemma characterising a rate for $\|\hat{V}_T-V\|_{\infty}$, which is used in the proofs of limiting distributions of our estimators. SM \ref{sec A6}, SM \ref{sec A7}, and SM \ref{sec A8} provide proofs of Theorem \ref{thm asymptotic normality MD when D is unknown}, Theorem \ref{prop Haihan score functions and second derivatives}, and Theorem \ref{thm one step estimator asymptotic normality}, respectively. SM \ref{sec A9} gives proofs of Theorem \ref{thm overidentification test fixed dim} and Corollary \ref{coro diagonal asymptotics}. SM \ref{secSM.cor.cramerwold} contains miscellaneous results.

\subsection{Additional Materials Related to the Kronecker Product}
\label{sec B1}

The following lemma proves a property of Kronecker products.

\begin{lemma}
\label{prop log kronecker}
Suppose $v=2,3,\ldots$ and that $A_1,A_2,\ldots,A_v$ are real symmetric and positive definite matrices of sizes $a_1\times a_1,\ldots, a_v\times a_v$, respectively. Then
\begin{align*}
&\log (A_1\otimes A_2\otimes \cdots \otimes A_v) \\
&= \log A_1\otimes I_{a_2}\otimes \cdots \otimes I_{a_v}+ I_{a_1} \otimes \log A_2\otimes I_{a_3}\otimes \cdots \otimes I_{a_v}+\cdots +I_{a_1}\otimes I_{a_2}\otimes \cdots \otimes \log A_v.
\end{align*}
\end{lemma}

\begin{proof}
We prove by mathematical induction. We first give a proof for $v=2$; that is,
\[\log (A_1\otimes A_2) \\
= \log A_1\otimes I_{a_2}+ I_{a_1} \otimes \log A_2.\]
Since $A_1,A_2$ are real symmetric, they can be orthogonally diagonalized: $A_i=U_{i}^{\intercal}\Lambda_{i}U_{i}$ for $i=1,2$, where $U_{i}$ is orthogonal, and $\Lambda_{i}=\text{diag}(\lambda_{i,1},\ldots,\lambda_{i,a_i})$ is a diagonal matrix containing those $a_i$ eigenvalues of $A_i$. Positive definiteness of $A_1,A_2$ ensures that their Kronecker product is positive definite. Then the logarithm of $A_1\otimes A_2$ is:
\begin{align}
& \log (A_1\otimes A_2) =\log
[(U_{1}\otimes U_2)^{\intercal}(\Lambda_{1} \otimes\Lambda_{2})(U_{1} \otimes U_2)]  =
(U_{1}\otimes U_2)^{\intercal}\log (\Lambda_{1} \otimes\Lambda_{2})(U_{1} \otimes U_2), \label{align log Sigma}%
\end{align}
where the first equality is due to the mixed product property of the Kronecker product, and the second equality is due to a property of matrix functions. Next,
\begin{align}
&  \log(\Lambda_{1}\otimes\Lambda_{2})=\text{diag}(\log(\lambda_{1,1}\Lambda
_{2}),\ldots,\log(\lambda_{1,a_1}\Lambda_{2}))  =\text{diag}(\log(\lambda_{1,1}I_{a_2}\Lambda_{2}),\ldots,\log
(\lambda_{1,a_{1}}I_{a_2}\Lambda_{2}))\nonumber\\
&  =\text{diag}(\log(\lambda_{1,1}I_{a_2})+\log(\Lambda_{2}),\ldots
,\log(\lambda_{1,a_{1}}I_{a_2})+\log(\Lambda_{2}))\nonumber\\
&  =\text{diag}(\log(\lambda_{1,1}I_{a_2}),\ldots,\log(\lambda_{1,a_{1}}%
I_{a_2}))+\text{diag}(\log(\Lambda_{2}),\ldots,\log(\Lambda_{2}))\nonumber\\
&  =\log(\Lambda_{1})\otimes I_{a_2}+I_{a_1}\otimes\log(\Lambda_{2}),\label{align log Lambda}
\end{align}
where the third equality holds only because $\lambda_{1,j}I_{a_2}$ and
$\Lambda_{2}$ have real positive eigenvalues only and commute for all
$j=1,\ldots,a_{1}$ (\cite{higham2008} p270 Theorem 11.3). Substitute
(\ref{align log Lambda}) into (\ref{align log Sigma}):
\begin{align*}
&\log (A_1\otimes A_2) =(U_{1}\otimes U_2)^{\intercal}\log (\Lambda_{1}\otimes\Lambda_{2})(U_{1}\otimes U_2)  =(U_{1}\otimes U_{2})^{\intercal}(\log\Lambda_{1}\otimes
I_{a_{2}}+I_{a_{1}}\otimes\log\Lambda_{2})(U_{1}\otimes U_{2})\\
&  =(U_{1}\otimes U_{2})^{\intercal}(\log\Lambda_{1}\otimes I_{a_2}%
)(U_{1}\otimes U_{2})+(U_{1}\otimes U_{2})^{\intercal}(I_{a_{1}}\otimes
\log\Lambda_{2})(U_{1}\otimes U_{2})\\
&  =\log A_1\otimes I_{a_{2}}+I_{a_{1}}\otimes\log A_{2}.
\end{align*}
We now assume that this lemma is true for $v=k$. That is,
\begin{align}
&\log (A_1\otimes A_2\otimes \cdots \otimes A_k) \notag \\
&= \log A_1\otimes I_{a_2}\otimes \cdots \otimes I_{a_k}+ I_{a_1} \otimes \log A_2\otimes I_{a_3}\otimes \cdots \otimes I_{a_k}+\cdots +I_{a_1}\otimes I_{a_2}\otimes \cdots \otimes \log A_k.\label{alignPktrue}
\end{align}
We prove that the lemma holds for $v=k+1$. Let $A_{1-k}:=A_1\otimes \cdots \otimes A_k$ and $I_{a_1\cdots a_k}:=I_{a_1}\otimes \cdots \otimes I_{a_k}$.
\begin{align*}
&\log (A_1\otimes A_2\otimes \cdots \otimes A_k\otimes A_{k+1})=\log (A_{1-k}\otimes A_{k+1})=\log A_{1-k}\otimes I_{a_{k+1}}+I_{a_1\cdots a_k}\otimes\log A_{k+1}\\
&=\log A_1\otimes I_{a_2}\otimes \cdots \otimes I_{a_k}\otimes I_{a_{k+1}}+ I_{a_1} \otimes \log A_2\otimes I_{a_3}\otimes \cdots \otimes I_{a_k}\otimes I_{a_{k+1}}+\cdots +\\
& \qquad I_{a_1}\otimes I_{a_2}\otimes \cdots \otimes \log A_k\otimes I_{a_{k+1}}+I_{a_1}\otimes \cdots \otimes I_{a_k}\otimes\log A_{k+1},
\end{align*} 
where the third equality is due to (\ref{alignPktrue}). Thus the lemma holds for $v=k+1$. By induction, the lemma is true for $v=2,3,\ldots$.
\end{proof}

\bigskip

Next we provide two examples to illustrate the necessity of an identification restriction in order to separately identify log parameters.

\begin{example}
\label{ex2x2}
Suppose that $n_{1}, n_{2}=2$. We have
\begin{align*}
%\log\Theta^{0}_{1}=\left(
%\begin{array}
%[c]{ccc}%
%0 & a_{1,2} & a_{1,3}\\
%a_{1,2} & a_{2,2} & a_{2,3}\\
%a_{1,3} & a_{2,3} & a_{3,3}%
%\end{array}
%\right)  \qquad
\log\Theta^{*}_{1}=\left(
\begin{array}
[c]{cc}%
a_{11} & a_{12}\\
a_{12} & a_{22}%
\end{array}
\right)  \qquad\log\Theta^{*}_{2}=\left(
\begin{array}
[c]{cc}%
b_{11} & b_{12}\\
b_{12} & b_{22}%
\end{array}
\right)
\end{align*}
Then we can calculate
\begin{align*}
\log \Theta^*=\log \Theta_1^*\otimes I_2 +I_2\otimes \log \Theta_2^*=\left( \begin{array}{cccc}
a_{11}+b_{11} & b_{12}              & a_{12}             & 0\\
b_{12}              & a_{11}+b_{22} & 0                      & a_{12}\\
a_{12}              & 0                       & a_{22}+b_{11} & b_{12}\\
0                       & a_{12}              & b_{12}              & a_{22}+b_{22}
\end{array}\right) .
\end{align*}
Log parameters $a_{12}, b_{12}$ can be separately identified from the off-diagonal entries of $\log \Theta^*$ because they appear separately. %In fact, $a_{12}$ appears alone in four off-diagonal entries of $\log \Theta^*$ so to estimate $a_{12}$ we will resort to some weighted average of these off-diagonal entries in the sample (see (\ref{eqn thetaTW closed form solution})). The same reasoning applies to $b_{12}$. 
We now examine whether log parameters $a_{11}, b_{11}, a_{22}, b_{22}$ can be separately identified from diagonal entries of  $\log \Theta^*$. The answer is no. We have the following linear system
\begin{align*}
Ax:=\left( \begin{array}{cccc}
1 & 0 & 1 & 0 \\
1 & 0 & 0 & 1\\
0 & 1 & 1 & 0\\
0 & 1 & 0 & 1
\end{array}\right) \left( 
\begin{array}{c}
a_{11}\\
a_{22}\\
b_{11}\\
b_{22}
\end{array}\right) =
\left( \begin{array}{c}
\sbr[1]{\log \Theta^*}_{11}\\
\sbr[1]{\log \Theta^*}_{22}\\
\sbr[1]{\log \Theta^*}_{33}\\
\sbr[1]{\log \Theta^*}_{44}
\end{array}\right) =:d.
\end{align*}
Note that the rank of $A$ is 3. There are three effective equations and four unknowns; the linear system has infinitely many solutions for $x$. Hence  one identification restriction is needed to separately identify log parameters $a_{11}, b_{11}, a_{22}, b_{22}$. We choose to set $a_{11}=0$.
\end{example}

\begin{example}
\label{ex2x2x2}
Suppose that $n_{1}, n_{2}, n_3=2$. We have
\begin{align*}
\log\Theta^{*}_{1}=\left(
\begin{array}
[c]{cc}%
a_{11} & a_{12}\\
a_{12} & a_{22}%
\end{array}
\right)  \qquad\log\Theta^{*}_{2}=\left(
\begin{array}
[c]{cc}%
b_{11} & b_{12}\\
b_{12} & b_{22}%
\end{array}
\right) \qquad\log\Theta^{*}_{3}=\left(
\begin{array}
[c]{cc}%
c_{11} & c_{12}\\
c_{12} & c_{22}%
\end{array}
\right)
\end{align*}
Then we can calculate
 \begin{align*}
&\log \Theta^*=\log \Theta_1^*\otimes I_2\otimes I_2 +I_2\otimes \log \Theta_2^*\otimes I_2+I_2\otimes I_2\otimes \log \Theta_3^*=\\
&{\tiny\left( \begin{array}{cccccccc}
a_{11}+b_{11}+c_{11} & c_{12} & b_{12} & 0       & a_{12}  & 0 & 0      & 0\\
c_{12}              & a_{11}+b_{11}+c_{22} & 0 & b_{12} & 0 & a_{12} & 0 & 0\\
b_{12}              & 0                       & a_{11}+b_{22}+c_{11} & c_{12} & 0 & 0 & a_{12} & 0\\
0     & b_{12}   & c_{12}     & a_{11}+b_{22}+c_{22} & 0 & 0 & 0 & a_{12}\\
a_{12} & 0 & 0 & 0 & a_{22}+b_{11}+c_{11} & c_{12} & b_{12} & 0\\
0 & a_{12} & 0 & 0 & c_{12} & a_{22}+b_{11}+c_{22} & 0 & b_{12}\\
0 & 0 & a_{12} & 0 & b_{12} & 0 & a_{22}+b_{22}+c_{11} & c_{12}\\
0 & 0 & 0 & a_{12} & 0 & b_{12} & c_{12} & a_{22}+b_{22}+c_{22}
\end{array}\right) }.
\end{align*}
Log parameters $a_{12}, b_{12},c_{12}$ can be separately identified from off-diagonal entries of $\log \Theta^*$ because they appear separately. We now examine whether log parameters $a_{11}, b_{11}, c_{11},a_{22}, b_{22},c_{22}$ can be separately identified from diagonal entries of  $\log \Theta^*$. The answer is no. We have the following linear system
\begin{align*}
Ax:=\left( \begin{array}{cccccc}
1 & 0 & 1 & 0 & 1 & 0 \\
1 & 0 & 1 & 0 & 0 & 1\\
1 & 0 & 0 & 1 & 1 & 0\\
1 & 0 & 0 & 1 & 0 & 1\\
0 & 1 & 1 & 0 & 1 & 0\\
0 & 1 & 1 & 0 & 0 & 1\\
0 & 1 & 0 & 1 & 1 & 0\\
0 & 1 & 0 & 1 & 0 & 1
\end{array}\right) \left( 
\begin{array}{c}
a_{11}\\
a_{22}\\
b_{11}\\
b_{22}\\
c_{11}\\
c_{22}
\end{array}\right) =
\left( \begin{array}{c}
\sbr[1]{\log \Theta^*}_{11}\\
\sbr[1]{\log \Theta^*}_{22}\\
\sbr[1]{\log \Theta^*}_{33}\\
\sbr[1]{\log \Theta^*}_{44}\\
\sbr[1]{\log \Theta^*}_{55}\\
\sbr[1]{\log \Theta^*}_{66}\\
\sbr[1]{\log \Theta^*}_{77}\\
\sbr[1]{\log \Theta^*}_{88}
\end{array}\right) =:d.
\end{align*}
Note that the rank of $A$ is 4. There are four effective equations and six unknowns; the linear system has infinitely many solutions for $x$. Hence two identification restrictions are needed to separately identify log parameters $a_{11}, b_{11}, c_{11},a_{22}, b_{22},c_{22}$. We choose to set $a_{11}=b_{11}=0$.
\end{example}

\subsection{Shrinkage via Minimum Distance}
\label{sec two stage MD}

Recall that in the fill and shrink method, there is no guarantee that the estimated $\exp (\log \Theta^0)$ will be a correlation matrix. However, the estimated $D^{1/2}\exp (\log \Theta^0)D^{1/2}$ will be a covariance matrix. As mentioned in the main article, one can re-normalise the estimated covariance matrix to obtain a correlation matrix. The alternative method would be to shrink $\exp (\log \Theta_j^0)$ to a correlation matrix for $j=1,\ldots,v$.

%If one insists to use the principle of minimum distance to estimate the correlation matrix, in order to make the final estimate a correlation matrix, we discuss a \textit{two-stage minimum distance} approach. 
This is easy for the $n=2^v$ case. Consider the $2\times 2$ submatrix $\Theta_1^0$, with $\log \Theta_1^0$ containing log parameters $\theta_1^0$. Given that $\Theta_1^0$ is a correlation matrix, then we have
\begin{align*}
\log\Theta_{1}^{0}  &  =\left(
\begin{array}
[c]{cc}%
1 & 1\\
1 & -1
\end{array}
\right)  \left(
\begin{array}
[c]{cc}%
\lambda_{1,1} & 0\\
0 & \lambda_{1,1}
\end{array}
\right)  \left(
\begin{array}
[c]{cc}%
1 & 1\\
1 & -1
\end{array}
\right)  \frac{1}{2}=\left( \begin{array}{cc}
\frac{1}{2}\lambda_{1,1}+\frac{1}{2}\lambda_{1,2} & \frac{1}{2}\lambda_{1,1}-\frac{1}{2}\lambda_{1,2}
 \\
\frac{1}{2}\lambda_{1,1}-\frac{1}{2}\lambda_{1,2} & \frac{1}{2}\lambda_{1,1}+\frac{1}{2}\lambda_{1,2}
\end{array}\right)
\end{align*} 
which implies that
\[\theta_1^0:=\left(\begin{array}{c}
\theta_{1,1}^0\\
\theta_{1,2}^0\\
\theta_{1,3}^0
\end{array} \right)=\frac{1}{2}\left(\begin{array}{cc}
1 & 1\\
1 & -1 \\
1 & 1
\end{array} \right)\left( \begin{array}{c}
\lambda_{1,1}\\
\lambda_{1,2}
\end{array}\right)=:C\left( \begin{array}{c}
\lambda_{1,1}\\
\lambda_{1,2}
\end{array}\right).   \]
Further, we have
\begin{align}
\Theta_{1}^{0}  &  =\left(
\begin{array}
[c]{cc}%
1 & 1\\
1 & -1
\end{array}
\right)  \left(
\begin{array}
[c]{cc}%
\exp(\lambda_{1,1}) & 0\\
0 & \exp(\lambda_{1,2})
\end{array}
\right)  \left(
\begin{array}
[c]{cc}%
1 & 1\\
1 & -1
\end{array}
\right)  \frac{1}{2}\notag\\
&=\left( \begin{array}{cc}
\frac{1}{2}\exp(\lambda_{1,1})+\frac{1}{2}\exp(\lambda_{1,2}) & \frac{1}{2}\exp(\lambda_{1,1})-\frac{1}{2}\exp(\lambda_{1,2})
 \\
\frac{1}{2}\exp(\lambda_{1,1})-\frac{1}{2}\exp(\lambda_{1,2}) & \frac{1}{2}\exp(\lambda_{1,1})+\frac{1}{2}\exp(\lambda_{1,2})
\end{array}\right).\label{align extra random 1}
\end{align}
By observing the diagonal elements of (\ref{align extra random 1}), we must have $\frac{1}{2}\exp(\lambda_{1,1})+\frac{1}{2}\exp(\lambda_{1,2})=1$ or equivalently $\lambda_{1,1}=\log \del[1]{2-\exp(\lambda_{1,2})}$. Also, we have
\begin{align}
\exp(\lambda_{1,1})-\exp(\lambda_{1,2})=2-2\exp(\lambda_{1,2})\in [-2,2],\label{align extra random 2}
\end{align}
by observing the off-diagonal elements of (\ref{align extra random 1}). From (\ref{align extra random 2}), we have $-\infty<\lambda_{1,2}\leq \log 2$.

%Recall that
%\begin{align*}
%\Theta^0&=\Theta_1^0\otimes \cdots \otimes \Theta_v^0\\
%\vech (\log \Theta^0)&=E\theta^0
%\end{align*}
%where $\theta^0$ is partitioned into $\theta_1^0,\ldots, \theta_v^0$. Likewise, we could partition the minimum distance estimator $\hat{\theta}_T$ into $\hat{\theta}_{T,1},\ldots, \hat{\theta}_{T,v}$.%; this corresponds to the first stage of the two-stage minimum distance estimation.

We now consider shrinkage. Given $\theta_1^0\in \mathbb{R}^3$ we define $\lambda_{1,2}$ as the solution of the following population objective function
\[\min_{t\in (-\infty,\log 2]}\left\|\theta_1^0-C\left( \begin{array}{c}
\log (2-\exp(t))\\
t
\end{array}\right)  \right\|_2 \]
Thus define the estimator $\hat{\lambda}_{1,2}$ to be the solution of the following sample objective function
\[\min_{t\in (-\infty,\log 2]}\left\|\hat{\theta}_{1}-C\left( \begin{array}{c}
\log (2-\exp(t))\\
t
\end{array}\right)  \right\|_2, \]
where $\hat{\theta}_{1}$ is some fill and shrink estimator of $\theta_1^0$. Then we calculate $\hat{\lambda}_{1,1}=\log (2-\exp(\hat{\lambda}_{1,2}))$. This ensures that $\hat{\Theta}_{1, S}^0:=\Theta_{1}^{0}(\hat{\lambda}_{1,1},\hat{\lambda}_{1,2})$ is a correlation matrix. We can repeat this procedure for other sub-matrices $\{\Theta_j^0\}_{j=2}^v$. The final estimate
\[\hat{\Theta}_{S}^0=\hat{\Theta}_{1, S}^0\otimes \cdots \otimes \hat{\Theta}_{v, S}^0\]
will be a correlation matrix. We acknowledge that for higher dimensional sub-matrices, this approach starts to get problematic. We leave it for future research. 

\subsection{A Rate for $\|\hat{V}_T-V\|_{\infty}$}
\label{sec A5}

The following lemma characterises a rate for $\|\hat{V}_T-V\|_{\infty}$, which is used in the proofs of limiting distributions of our estimators.

\begin{lemma}
\label{lemma rate for hatV-V infty}
Let Assumptions \ref{assu subgaussian vector}(i) and \ref{assu mixing} be satisfied with $1/\gamma:=1/r_1+1/r_2>1$. Suppose $\log n=o(T^{\frac{\gamma}{2-\gamma}})$ if $n>T$.
Then
\[\|\hat{V}_{T}-V\|_{\infty}=O_{p}\del[3]{\sqrt{\frac{\log n}{T}}}.\]
\end{lemma}

\begin{proof}
Let $\tilde{y}_{t,i}$ denote $y_{t,i}-\bar{y}_i$, similarly for $\tilde{y}_{t,j},\tilde{y}_{t,k},\tilde{y}_{t,\ell}$, where $i,j,k,\ell=1,\ldots,n$. Let $\dot{y}_{t,i}$ denote $y_{t,i}-\mu_i$, similarly for $\dot{y}_{t,j},\dot{y}_{t,k},\dot{y}_{t,\ell}$ where $i,j,k,\ell=1,\ldots,n$.
\begin{align}
&\|\hat{V}_T-V\|_{\infty}:=\max_{1\leq a,b\leq n^2}|\hat{V}_{T,a,b}-V_{a,b}|=\max_{1\leq i,j,k,\ell\leq n}|\hat{V}_{T,i,j,k,\ell}-V_{i,j,k,\ell}| \notag\\
&\leq  \max_{1\leq i,j,k,\ell\leq n}\envert[3]{ \frac{1}{T}\sum_{t=1}^{T}\tilde{y}_{t,i}\tilde{y}_{t,j}\tilde{y}_{t,k}\tilde{y}_{t,\ell}-\frac{1}{T}\sum_{t=1}^{T}\dot{y}_{t,i}\dot{y}_{t,j}\dot{y}_{t,k}\dot{y}_{t,\ell}} \label{align hatV-V infinity part C}\\
&+ \max_{1\leq i,j,k,\ell\leq n}\envert[3]{ \frac{1}{T}\sum_{t=1}^{T}\dot{y}_{t,i}\dot{y}_{t,j}\dot{y}_{t,k}\dot{y}_{t,\ell}-\mathbb{E}[\dot{y}_{t,i}\dot{y}_{t,j}\dot{y}_{t,k}\dot{y}_{t,\ell}]} \label{align hatV-V infinity part A}\\
& +\max_{1\leq i,j,k,\ell\leq n}\envert[3]{\del[3]{ \frac{1}{T}\sum_{t=1}^{T}\tilde{y}_{t,i}\tilde{y}_{t,j}} \del[3]{ \frac{1}{T}\sum_{t=1}^{T}\tilde{y}_{t,k}\tilde{y}_{t,\ell}}-\del[3]{ \frac{1}{T}\sum_{t=1}^{T}\dot{y}_{t,i}\dot{y}_{t,j}} \del[3]{ \frac{1}{T}\sum_{t=1}^{T}\dot{y}_{t,k}\dot{y}_{t,\ell}}}  \label{align hatV-V infinity part D}\\
& +\max_{1\leq i,j,k,\ell\leq n}\envert[3]{\del[3]{ \frac{1}{T}\sum_{t=1}^{T}\dot{y}_{t,i}\dot{y}_{t,j}} \del[3]{ \frac{1}{T}\sum_{t=1}^{T}\dot{y}_{t,k}\dot{y}_{t,\ell}}-\mathbb{E}[\dot{y}_{t,i}\dot{y}_{t,j}]\mathbb{E}[\dot{y}_{t,k}\dot{y}_{t,\ell}]}  \label{align hatV-V infinity part B}
\end{align}

\subsubsection*{Display (\ref{align hatV-V infinity part A})}

Assumption \ref{assu subgaussian vector}(i) says that for all $t$, there exist absolute constants $K_1>1, K_2>0, r_1>0$ such that
\[\mathbb{E}\sbr[2]{\exp\del [1]{K_2 |y_{t,i}|^{r_1}}}\leq K_1\qquad \text{for all }i=1,\ldots,n.\]
By repeated using Lemma \ref{lemmaexponentialtail} in Appendix \ref{secArateofconvergence}, we have for all $i,j,k,\ell=1,2,\ldots,n$, every $\epsilon\geq 0$, absolute constants $b_1,c_1,b_2,c_2,b_3,c_3>0$ such that
\begin{align*}
\mathbb{P}(|y_{t,i}|\geq \epsilon)& \leq \exp \sbr[1]{1-(\epsilon/b_1)^{r_1}}\\
\mathbb{P}(|\dot{y}_{t,i}|\geq \epsilon)& \leq \exp \sbr[1]{1-(\epsilon/c_1)^{r_1}}\\
\mathbb{P}(|\dot{y}_{t,i}\dot{y}_{t,j}|\geq \epsilon)& \leq \exp \sbr[1]{1-(\epsilon/b_2)^{r_3}}\\
\mathbb{P}(|\dot{y}_{t,i}\dot{y}_{t,j}-\mathbb{E}[\dot{y}_{t,i}\dot{y}_{t,j}]|\geq \epsilon)& \leq \exp \sbr[1]{1-(\epsilon/c_2)^{r_3}}\\
\mathbb{P}(|\dot{y}_{t,i}\dot{y}_{t,j}\dot{y}_{t,k}\dot{y}_{t,\ell}|\geq \epsilon)& \leq \exp \sbr[1]{1-(\epsilon/b_3)^{r_4}}\\
\mathbb{P}(|\dot{y}_{t,i}\dot{y}_{t,j}\dot{y}_{t,k}\dot{y}_{t,\ell}-\mathbb{E}[\dot{y}_{t,i}\dot{y}_{t,j}\dot{y}_{t,k}\dot{y}_{t,\ell}]|\geq \epsilon)& \leq \exp \sbr[1]{1-(\epsilon/c_3)^{r_4}}
\end{align*}
where $r_3\in (0, r_1/2]$ and $r_4\in (0, r_1/4]$. Use the assumption $1/r_1+1/r_2>1$ to invoke Theorem \ref{thmbernsteininequality} followed by Lemma \ref{lemmabernsteinrate} in Appendix \ref{sec oldappendixB} to get

%By Assumption \ref{assu subgaussian vector}(i), $y_{t,i}, y_{t,j}, y_{t,k}, y_{t,\ell}$ are subgaussian random variables. We now show that $\dot{y}_{t,i}, \dot{y}_{t,j}, \dot{y}_{t,k},\dot{y}_{t,\ell}$ are also uniformly subgaussian. Without loss of generality consider $\dot{y}_{t,i}$.
%\begin{align*}
%&\mathbb{P}\del [1]{|\dot {x}_{t,i}|\geq \epsilon}=\mathbb{P}\del [1]{|y_{t,i}-\mu_i|\geq \epsilon}\leq \mathbb{P}\del [1]{|y_{t,i}|\geq \epsilon-|\mu_i|}\leq K e^{-C(\epsilon-|\mu_i|)^2}\\
%&\leq Ke^{-C\epsilon^2}e^{2C\epsilon |\mu_i|}e^{-C|\mu_i|^2}\leq Ke^{-C\epsilon^2}e^{2C\epsilon |\mu_i|}\leq Ke^{-C\epsilon^2}e^{C(\epsilon^2/2+2|\mu_i|^2)}\\
%&=Ke^{-\frac{1}{2}C\epsilon^2}e^{2C|\mu_i|^2}\leq Ke^{-\frac{1}{2}C\epsilon^2}e^{2C(\max_{1\leq i\leq n}|\mu_i|)^2}=K'e^{-\frac{1}{2}C\epsilon^2},
%\end{align*}
%where the fifth inequality is due to the decoupling inequality $2xy\leq x^2/2+2y^2$, and the last equality is due to (\ref{align mu Op1}). We now consider (\ref{align hatV-V infinity part A}). Invoke Proposition \ref{prop product of L subgaussian} in Appendix \ref{sec oldappendixB}:
%
\begin{equation}
\label{eqn rate for hatV-V infinity part A}
\max_{1\leq i,j,k,\ell\leq n}\envert[3]{ \frac{1}{T}\sum_{t=1}^{T}\dot{y}_{t,i}\dot{y}_{t,j}\dot{y}_{t,k}\dot{y}_{t,\ell}-\mathbb{E}\dot{y}_{t,i}\dot{y}_{t,j}\dot{y}_{t,k}\dot{y}_{t,\ell}}=O_p\del[3]{ \sqrt{\frac{\log n}{T}}}.
\end{equation}

\subsubsection*{Display (\ref{align hatV-V infinity part B})}
We now consider (\ref{align hatV-V infinity part B}).
\begin{align}
&\max_{1\leq i,j,k,\ell\leq n}\envert[3]{\del[3]{ \frac{1}{T}\sum_{t=1}^{T}\dot{y}_{t,i}\dot{y}_{t,j}} \del[3]{ \frac{1}{T}\sum_{t=1}^{T}\dot{y}_{t,k}\dot{y}_{t,\ell}}-\mathbb{E}[\dot{y}_{t,i}\dot{y}_{t,j}]\mathbb{E}[\dot{y}_{t,k}\dot{y}_{t,\ell}]}\notag\\
&\leq \max_{1\leq i,j,k,\ell\leq n}\envert[3]{\del[3]{ \frac{1}{T}\sum_{t=1}^{T}\dot{y}_{t,i}\dot{y}_{t,j}}\del[3]{ \frac{1}{T}\sum_{t=1}^{T}\dot{y}_{t,k}\dot{y}_{t,\ell}-\mathbb{E}[\dot{y}_{t,k}\dot{y}_{t,\ell}]}}\label{align hatV-V infinity part Ba}\\
&\qquad+\max_{1\leq i,j,k,\ell\leq n}\envert[3]{\mathbb{E}[\dot{y}_{t,k}\dot{y}_{t,\ell}]\del[3]{ \frac{1}{T}\sum_{t=1}^{T}\dot{y}_{t,i}\dot{y}_{t,j}-\mathbb{E}[\dot{y}_{t,i}\dot{y}_{t,j}]} }\label{align hatV-V infinity part Bb}.
\end{align}
Consider (\ref{align hatV-V infinity part Ba}).
\begin{align*}
&\max_{1\leq i,j,k,\ell\leq n}\envert[3]{\del[3]{ \frac{1}{T}\sum_{t=1}^{T}\dot{y}_{t,i}\dot{y}_{t,j}} \del[3]{ \frac{1}{T}\sum_{t=1}^{T}\dot{y}_{t,k}\dot{y}_{t,\ell}-\mathbb{E}\dot{y}_{t,k}\dot{y}_{t,\ell}}}\notag\\
&\leq \max_{1\leq i,j\leq n}\del[3]{ \envert[3]{ \frac{1}{T}\sum_{t=1}^{T}\dot{y}_{t,i}\dot{y}_{t,j}-\mathbb{E}\dot{y}_{t,i}\dot{y}_{t,j}} +\left| \mathbb{E}\dot{y}_{t,i}\dot{y}_{t,j}\right|}  \max_{1\leq k,\ell\leq n} \envert[3]{ \frac{1}{T}\sum_{t=1}^{T}\dot{y}_{t,k}\dot{y}_{t,\ell}-\mathbb{E}\dot{y}_{t,k}\dot{y}_{t,\ell}}\notag\\
&=\del[3]{ O_p\del[3]{ \sqrt{\frac{\log n}{T}}}+O(1)} O_p\del[3]{ \sqrt{\frac{\log n}{T}}}=O_p\del[3]{ \sqrt{\frac{\log n}{T}}}
\end{align*}
where the first equality is due to Lemma \ref{lemmaexponentialtail}(ii) in Appendix \ref{secArateofconvergence}, Theorem \ref{thmbernsteininequality} and Lemma \ref{lemmabernsteinrate} in Appendix \ref{sec oldappendixB}. Now consider (\ref{align hatV-V infinity part Bb}).
\begin{align*}
&\max_{1\leq i,j,k,\ell\leq n}\envert[3]{\mathbb{E}[\dot{y}_{t,k}\dot{y}_{t,\ell}]\del[3]{ \frac{1}{T}\sum_{t=1}^{T}\dot{y}_{t,i}\dot{y}_{t,j}-\mathbb{E}[\dot{y}_{t,i}\dot{y}_{t,j}]}} \notag\\
&\leq \max_{1\leq k,\ell\leq n}|\mathbb{E}[\dot{y}_{t,k}\dot{y}_{t,\ell}]|\max_{1\leq i,j\leq n}\envert[3]{ \frac{1}{T}\sum_{t=1}^{T}\dot{y}_{t,i}\dot{y}_{t,j}-\mathbb{E}\dot{y}_{t,i}\dot{y}_{t,j}}=O_p\del[3]{ \sqrt{\frac{\log n}{T}}}
\end{align*}
where the equality is due to Lemma \ref{lemmaexponentialtail}(ii) in Appendix \ref{secArateofconvergence}, Theorem \ref{thmbernsteininequality} and Lemma \ref{lemmabernsteinrate} in Appendix \ref{sec oldappendixB}. Thus
\begin{equation}
\label{align rate for hatV-V infinity part Ba}
\max_{1\leq i,j,k,\ell\leq n}\envert[3]{\del[3]{ \frac{1}{T}\sum_{t=1}^{T}\dot{y}_{t,i}\dot{y}_{t,j}} \del[3]{ \frac{1}{T}\sum_{t=1}^{T}\dot{y}_{t,k}\dot{y}_{t,\ell}}-\mathbb{E}[\dot{y}_{t,i}\dot{y}_{t,j}]\mathbb{E}[\dot{y}_{t,k}\dot{y}_{t,\ell}]}=O_p\del[3]{ \sqrt{\frac{\log n}{T}}}.
\end{equation}

\subsubsection*{Display (\ref{align hatV-V infinity part C})}

We first give a rate for $\max_{1\leq i\leq n}|\bar{y}_i-\mu_i|$. The index $i$ is arbitrary and could be replaced with $j,k,\ell$. %By Assumption \ref{assu subgaussian vector}(i), $\{y_{t,i}\}_{t=1}^T$ are independent subgaussian random variables. For $\epsilon>0$, $\mathbb{P}(|y_{t,i}|\geq \epsilon)\leq Ke^{-C\epsilon^2}$. It follows from Lemma 2.2.1 in \cite{vandervaartWellner1996} that $\|y_{t,i}\|_{\psi_2}\leq (1+K)^{1/2}/C^{1/2}$. Then $\|y_{t,i}-\mathbb{E}y_{t,i}\|_{\psi_2}\leq\|y_{t,i}\|_{\psi_2}+\mathbb{E}\|y_{t,i}\|_{\psi_2}\leq \frac{2(1+K)^{1/2}}{C^{1/2}}$. Next, using the second last inequality in \cite{vandervaartWellner1996} p95, we have
%\[\|y_{t,i}-\mathbb{E}y_{t,i}\|_{\psi_1}\leq \|y_{t,i}-\mathbb{E}y_{t,i}\|_{\psi_2}(\log 2)^{-1/2}\leq  \frac{2(1+K)^{1/2}}{C^{1/2}}(\log 2)^{-1/2}=:\frac{1}{W}.\]
%
%Then, by the definition of the Orlicz norm, $\mathbb{E}\sbr[1]{ e^{W|y_{t,i}-\mathbb{E}y_{t,i}|} }\leq 2$. Use Fubini's theorem to expand out the exponential moment. It is easy to see that $y_{t,i}-\mathbb{E}y_{t,i}$ satisfies the moment conditions of Bernstein's inequality in Appendix \ref{sec oldappendixB} with $A=\frac{1}{W}$ and $\sigma_0^2=\frac{2}{W^2}$. Now invoke Bernstein's inequality for all $\epsilon>0$
%
%\[\mathbb{P}\del[3]{ \envert[3]{ \frac{1}{T}\sum_{t=1}^{T}(y_{t,i}-\mathbb{E}y_{t,i})}\geq \sigma_0^2\left[ A\epsilon+\sqrt{2\epsilon}\right] }\leq 2e^{-T\sigma_0^2\epsilon}. \]
Invoking Lemma \ref{lemmabernsteinrate} in Appendix \ref{sec oldappendixB}, we have
\begin{equation}
\label{eqn barxi minus mui}
\max_{1\leq i\leq n}|\bar{y}_i-\mu_i|=\max_{1\leq i\leq n}\envert[3]{\frac{1}{T}\sum_{t=1}^{T}(y_{t,i}-\mu_i)}=O_p\del [3]{ \sqrt{\frac{\log n}{T}}}.
\end{equation}
Then we also have
\begin{equation}
\label{eqn barxi Op1}
\max_{1\leq i\leq n}|\bar{y}_i|=\max_{1\leq i\leq n}|\bar{y}_i-\mu_i+\mu_i|\leq \max_{1\leq i\leq n}|\bar{y}_i-\mu_i|+\max_{1\leq i\leq n}|\mu_i|=O_p\del [3]{ \sqrt{\frac{\log n}{T}}}+O(1)=O_p(1).
\end{equation}

We now consider (\ref{align hatV-V infinity part C}):
\begin{align*}
 \max_{1\leq i,j,k,\ell\leq n}\envert[3]{ \frac{1}{T}\sum_{t=1}^{T}\tilde{y}_{t,i}\tilde{y}_{t,j}\tilde{y}_{t,k}\tilde{y}_{t,\ell}-\frac{1}{T}\sum_{t=1}^{T}\dot{y}_{t,i}\dot{y}_{t,j}\dot{y}_{t,k}\dot{y}_{t,\ell}}.
\end{align*}
With expansion, simplification and recognition that the indices $i,j,k,\ell$ are completely symmetric, we can bound (\ref{align hatV-V infinity part C}) by
\begin{align}
& \max_{1\leq i,j,k,\ell\leq n}\envert{\bar{y}_i\bar{y}_j\bar{y}_k\bar{y}_{\ell}-\mu_i\mu_j\mu_k\mu_{\ell}}\label{align 4-1}\\
&+4 \max_{1\leq i,j,k,\ell\leq n} \envert[2]{\bar{y}_i\del[1]{\bar{y}_j\bar{y}_k\bar{y}_{\ell}-\mu_j\mu_k\mu_{\ell}}}\label{align 4-2}\\
&+6 \max_{1\leq i,j,k,\ell\leq n} \envert[3]{\del[3]{\frac{1}{T}\sum_{t=1}^{T}y_{t,i}y_{t,j}}\del[1]{\bar{y}_k\bar{y}_{\ell}-\mu_k\mu_{\ell}}}\label{align 4-3}\\
&+4 \max_{1\leq i,j,k,\ell\leq n} \envert[3]{\del[3]{\frac{1}{T}\sum_{t=1}^{T}y_{t,i}y_{t,j}y_{t,k}}\del[1]{\bar{y}_{\ell}-\mu_{\ell}}}\label{align 4-4}.
\end{align}
We consider (\ref{align 4-1}) first. (\ref{align 4-1}) can be bounded by repeatedly invoking triangular inequalities (e.g., inserting terms like $\mu_i\bar{y}_j\bar{y}_k\bar{y}_{\ell}$) using Lemma \ref{lemmaexponentialtail}(ii) in Appendix \ref{secArateofconvergence}, (\ref{eqn barxi Op1})  and (\ref{eqn barxi minus mui}). (\ref{align 4-1}) is of order $O_p(\sqrt{\log n/T})$.  (\ref{align 4-2}) is of order $O_p(\sqrt{\log n/T})$ by a similar argument.  (\ref{align 4-3}) and  (\ref{align 4-4}) are of the same order $O_p(\sqrt{\log n/T})$ using a similar argument provided that both $\max_{1\leq i,j\leq n}|\sum_{t=1}^{T}y_{t,i}y_{t,j}|/T$ and $\max_{1\leq i,j,k\leq n}|\sum_{t=1}^{T}y_{t,i}y_{t,j}y_{t,k}|/T$ are $O_p(1)$; these follow from Lemma \ref{lemmaexponentialtail}(ii) in Appendix \ref{secArateofconvergence}, Theorem \ref{thmbernsteininequality} and Lemma \ref{lemmabernsteinrate} in Appendix \ref{sec oldappendixB}. Thus
\begin{equation}
\label{eqn rate for hatV-V infinity part C}
 \max_{1\leq i,j,k,\ell\leq n}\envert[3]{ \frac{1}{T}\sum_{t=1}^{T}\tilde{y}_{t,i}\tilde{y}_{t,j}\tilde{y}_{t,k}\tilde{y}_{t,\ell}-\frac{1}{T}\sum_{t=1}^{T}\dot{y}_{t,i}\dot{y}_{t,j}\dot{y}_{t,k}\dot{y}_{t,\ell}}=O_p(\sqrt{\log n/T}).
\end{equation}

\subsubsection*{Display (\ref{align hatV-V infinity part D})}

We now consider (\ref{align hatV-V infinity part D}).
\begin{align}
& \max_{1\leq i,j,k,\ell\leq n}\envert[3]{\del[3]{ \frac{1}{T}\sum_{t=1}^{T}\tilde{y}_{t,i}\tilde{y}_{t,j}} \del[3]{ \frac{1}{T}\sum_{t=1}^{T}\tilde{y}_{t,k}\tilde{y}_{t,\ell}}-\del[3]{ \frac{1}{T}\sum_{t=1}^{T}\dot{y}_{t,i}\dot{y}_{t,j}} \del[3]{ \frac{1}{T}\sum_{t=1}^{T}\dot{y}_{t,k}\dot{y}_{t,\ell}}} \notag \\
&\leq \max_{1\leq i,j,k,\ell\leq n}\envert[3]{\del[3]{ \frac{1}{T}\sum_{t=1}^{T}\tilde{y}_{t,i}\tilde{y}_{t,j}} \del[3]{ \frac{1}{T}\sum_{t=1}^{T}\del[1]{\tilde{y}_{t,k}\tilde{y}_{t,\ell}-\dot{y}_{t,k}\dot{y}_{t,\ell}}}}\label{align partDa}\\
&\qquad+\max_{1\leq i,j,k,\ell\leq n}\envert[3]{\del[3]{ \frac{1}{T}\sum_{t=1}^{T}\dot{y}_{t,k}\dot{y}_{t,\ell}} \del[3]{ \frac{1}{T}\sum_{t=1}^{T}\del[1]{\tilde{y}_{t,i}\tilde{y}_{t,j}-\dot{y}_{t,i}\dot{y}_{t,j}}}}\label{align partDb}
\end{align}
It suffices to give a bound for (\ref{align partDa}) as the bound for (\ref{align partDb}) is of the same order and follows through similarly. First, it is easy to show that $\max_{1\leq i,j\leq n}|\frac{1}{T}\sum_{t=1}^{T}\tilde{y}_{t,i}\tilde{y}_{t,j}|=\max_{1\leq i,j\leq n}|\frac{1}{T}\sum_{t=1}^{T}y_{t,i}y_{t,j}-\bar{y}_i\bar{y}_j|=O_p(1)$ (using Lemma \ref{lemmaexponentialtail}(ii) in Appendix \ref{secArateofconvergence} and Lemma \ref{lemmabernsteinrate} in Appendix \ref{sec oldappendixB}). Next
\begin{align}
& \max_{1\leq k,\ell\leq n}\envert[3]{  \frac{1}{T}\sum_{t=1}^{T}\del[1]{\tilde{y}_{t,k}\tilde{y}_{t,\ell}-\dot{y}_{t,k}\dot{y}_{t,\ell}}}= \max_{1\leq k,\ell\leq n}\envert[3]{  -(\bar{y}_k-\mu_k)(\bar{y}_{\ell}-\mu_{\ell})}=O_p\del [3]{\frac{\log n}{T}}.\label{eqn rate for hatV-V infinity part D}
\end{align}

The lemma follows after summing up the rates for (\ref{eqn rate for hatV-V infinity part A}), (\ref{align rate for hatV-V infinity part Ba}), (\ref{eqn rate for hatV-V infinity part C}) and (\ref{eqn rate for hatV-V infinity part D}).
\end{proof}

\subsection{Proof of Theorem \ref{thm asymptotic normality MD when D is unknown}}
\label{sec A6}

In this subsection, we give a proof for Theorem \ref{thm asymptotic normality MD when D is unknown}. We will first give a preliminary lemma leading to the proof of this theorem.

\begin{lemma}
Let Assumptions \ref{assu subgaussian vector}(i), \ref{assu mixing}, \ref{assu n indexed by T}(i) and  \ref{assu about D and Dhat}(i) hold with $1/r_1+1/r_2>1$. Then we have 
\begin{equation}
\label{eqn spectral norm of P and Phat}
\|P\|_{\ell_{2}}=O(1),\qquad\|\hat{P}_{T}\|_{\ell_{2}}=O_{p}(1),\qquad\|\hat{P}_{T}-P\|_{\ell_{2}}=O_{p}\del[3]{ \sqrt{\frac{n}{T}}}.
\end{equation}
\end{lemma}

\begin{proof}
The proofs for $\|P\|_{\ell_2}=O(1)$ and $\|\hat{P}_T\|_{\ell_2}=O_p(1)$ are exactly the same, so we only give the proof for the latter.
\begin{align*}
& \|\hat{P}_T\|_{\ell_2}=\|I_{n^2}-D_nD_n^{+}(I_n\otimes\hat{\Theta}_T)M_d\|_{\ell_2}\leq 1+\|D_nD_n^{+}(I_n\otimes\hat{\Theta}_T)M_d\|_{\ell_2}\\
&\leq  1+\|D_n\|_{\ell_2}\|D_n^{+}\|_{\ell_2}\|I_n\otimes\hat{\Theta}_T\|_{\ell_2}\|M_d\|_{\ell_2}=1+2\|I_n\|_{\ell_2}\|\hat{\Theta}_T\|_{\ell_2}=O_p(1)
\end{align*}
where the second equality is due to (\ref{eqn spectral norm for Dnplus and Dn}) and Lemma \ref{lemma l2norm of kronecker product} in Appendix \ref{sec oldappendixB}, and last equality is due to Lemma \ref{prop mini eigenvalue}(ii). Now,
\begin{align*}
&\|\hat{P}_T-P\|_{\ell_2}=\|I_{n^2}-D_nD_n^{+}(I_n\otimes\hat{\Theta}_T)M_d-(I_{n^2}-D_nD_n^{+}(I_n\otimes\Theta)M_d)\|_{\ell_2}\\
&=\|D_nD_n^{+}(I_n\otimes\hat{\Theta}_T)M_d-D_nD_n^{+}(I_n\otimes\Theta)M_d)\|_{\ell_2}=\|D_nD_n^{+}(I_n\otimes(\hat{\Theta}_T-\Theta))M_d\|_{\ell_2}\\
&=O_p(\sqrt{n/T}),
\end{align*}
where the last equality is due to Theorem \ref{thm main rate of convergence}(i).
\end{proof}

\bigskip

We are now ready to give a poof for Theorem \ref{thm asymptotic normality MD when D is unknown}.

\begin{proof}[Proof of Theorem \ref{thm asymptotic normality MD when D is unknown}]
We write
\begin{align*}
&\frac{\sqrt{T}c^{\intercal}(\hat{\theta}_{T}-\theta^0)}{\sqrt{c^{\intercal}\hat{J}_{T}c}}\\
&=\frac{\sqrt{T}c^{\intercal}(E^{\intercal}WE)^{-1}E^{\intercal}WD_n^+H\ve (\hat{\Theta}_T-\Theta)}{\sqrt{c^{\intercal}\hat{J}_{T}c}}+\frac{\sqrt{T}c^{\intercal}(E^{\intercal}WE)^{-1}E^{\intercal}WD_n^+\ve O_p(\|\hat{\Theta}_{T}-\Theta\|_{\ell_2}^2)}{\sqrt{c^{\intercal}\hat{J}_{T}c}}\\
&=\frac{\sqrt{T}c^{\intercal}(E^{\intercal}WE)^{-1}E^{\intercal}WD_n^+H\left. \frac{\partial \ve \Theta}{\partial \ve \Sigma}\right| _{\Sigma=\mathring{\Sigma}_T^{(i)}}\ve (\hat{\Sigma}_T-\Sigma)}{\sqrt{c^{\intercal}\hat{J}_{T}c}}\\
&\qquad+\frac{\sqrt{T}c^{\intercal}(E^{\intercal}WE)^{-1}E^{\intercal}WD_n^+\ve O_p(\|\hat{\Theta}_{T}-\Theta\|_{\ell_2}^2)}{\sqrt{c^{\intercal}\hat{J}_{T}c}}\\
& =:\hat{t}_{1}+\hat{t}_{2},
\end{align*}
where $\left. \frac{\partial \ve \Theta}{\partial \ve \Sigma}\right| _{\Sigma=\mathring{\Sigma}_T^{(i)}}$ denotes a matrix whose $j$th row is the $j$th row of the Jacobian matrix $\frac{\partial \ve \Theta}{\partial \ve \Sigma}$ evaluated at $\ve \mathring{\Sigma}^{(j)}_T$, which is a point between $\ve\Sigma$ and $\ve\hat{\Sigma}_T$, for $j=1,\ldots,n^2$.

Define
\[t_{1}:=\frac{\sqrt{T}c^{\intercal}(E^{\intercal}WE)^{-1}E^{\intercal}WD_n^+HP(D^{-1/2}\otimes D^{-1/2})\ve (\tilde{\Sigma}_T-\Sigma)}{\sqrt{c^{\intercal}Jc}}.\]
To prove Theorem \ref{thm asymptotic normality MD when D is unknown}, it suffices to show $t_{1}\xrightarrow{d}N(0,1)$, $t_{1}-\hat{t}_{1}=o_p(1)$, and $\hat{t}_{2}=o_p(1)$. The proof is similar to that of Theorem \ref{thm asymptotic normality}, so we will be concise for the parts which are almost identical to those of Theorem \ref{thm asymptotic normality}.

\subsubsection{$t_{1}\xrightarrow{d}N(0,1)$}

We now prove that $t_{1}$ is asymptotically distributed as a standard normal.
\begin{align*}
& t_{1}=\\
&\frac{\sqrt{T}c^{\intercal}(E^{\intercal}WE)^{-1}E^{\intercal}WD_n^+HP(D^{-1/2}\otimes D^{-1/2})\ve \del[2] {\frac{1}{T}\sum_{t=1}^{T}\sbr[1]{ (y_t-\mu)(y_t-\mu)^{\intercal}-\mathbb{E}(y_t-\mu)(y_t-\mu)^{\intercal}}}}{\sqrt{c^{\intercal}Jc}}\\
&=\sum_{t=1}^{T}\frac{T^{-1/2}c^{\intercal}(E^{\intercal}WE)^{-1}E^{\intercal}WD_n^+HP(D^{-1/2}\otimes D^{-1/2})\ve \sbr[1]{ (y_t-\mu)(y_t-\mu)^{\intercal}-\mathbb{E}(y_t-\mu)(y_t-\mu)^{\intercal}}}{\sqrt{c^{\intercal}Jc}}\\
&=:\sum_{t=1}^{T}U_{T,n,t}.
\end{align*}
%
%Trivially $\mathbb{E}[U_{T,n,t}]=0$ and $\sum_{t=1}^{T}\mathbb{E}[U_{T,n,t}^2]=1$. Then we just need to verify the following Lindeberg condition for a double indexed process (e.g., \cite{phillipsmoon1999} Theorem 2 p1070): for all $\varepsilon>0$,
%\[\lim_{n,T\to \infty}\sum_{t=1}^{T}\int_{\{|U_{T,n,t}|\geq \varepsilon\}}U_{T,n,t}^2dP=0.\]
%For any $\gamma>2$,
%\begin{align*}
%&\int_{\{|U_{T,n,t}|\geq \varepsilon\}}U_{T,n,t}^2dP=\int_{\{|U_{T,n,t}|\geq \varepsilon\}}U_{T,n,t}^2|U_{T,n,t}|^{-\gamma}|U_{T,n,t}|^{\gamma}dP\leq \varepsilon^{2-\gamma}\int_{\{|U_{T,n,t}|\geq \varepsilon\}}|U_{T,n,t}|^{\gamma}dP\\
%&\leq  \varepsilon^{2-\gamma}\mathbb{E}|U_{T,n,t}|^{\gamma}.
%\end{align*}
%
Again it is straightforward to show that $\{U_{T,n,t},\mathcal{F}_{T,n,t}\}$ is a martingale difference sequence. We first investigate at what rate the denominator $\sqrt{c^{\intercal}Jc}$ goes to zero:
\begin{align*}
& c^{\intercal}Jc=c^{\intercal}(E^{\intercal}WE)^{-1}E^{\intercal}WD_{n}^{+}HP(D^{-1/2}\otimes D^{-1/2})V(D^{-1/2}\otimes D^{-1/2})P^{\intercal}HD_{n}^{+^{\intercal}}WE(E^{\intercal}WE)^{-1}c\\
&\geq \text{mineval}\del [1]{E^{\intercal}WD_{n}^{+}HP(D^{-1/2}\otimes D^{-1/2})V(D^{-1/2}\otimes D^{-1/2})P^{\intercal}HD_{n}^{+^{\intercal}}WE}\|(E^{\intercal}WE)^{-1}c\|_2^2\\
&\geq \frac{n}{\varpi}\text{mineval}^2(W)c(E^{\intercal}WE)^{-2}c\geq \frac{n}{\varpi}\text{mineval}^2(W)\text{mineval}\del [1]{(E^{\intercal}WE)^{-2}}\\
&=\frac{n\cdot \text{mineval}^2(W)}{\varpi\text{maxeval}^2(E^{\intercal}WE)}\geq \frac{n}{\varpi \text{maxeval}^2(W^{-1})\text{maxeval}^2(W)\text{maxeval}^2(E^{\intercal}E)}\\
&=\frac{n}{\varpi \kappa^2(W)\text{maxeval}^2(E^{\intercal}E)}
\end{align*}
where the second inequality is due to Assumption \ref{assu when D is unknown}(ii). Using (\ref{eqn maxeval EE}), we have
\begin{equation}
\label{eqn inverse square root G rate when D is unknown}
\frac{1}{\sqrt{c^{\intercal}Jc}}=O(\sqrt{s^2\cdot n\cdot \kappa^2(W)\cdot \varpi}).
\end{equation}
%
%Then a sufficient condition for the Lindeberg condition is:
%\begin{align}
%& T^{1-\frac{\gamma}{2}}(s^2n\kappa^2(W)\varpi)^{\gamma/2}\notag \\
%&  \cdot \mathbb{E} \left| c^{\intercal}(E^{\intercal}WE)^{-1}E^{\intercal}WD^+_nHP(D^{-1/2}\otimes D^{-1/2})\ve \sbr[1]{(y_t-\mu)(y_t-\mu)^{\intercal}-\mathbb{E}(y_t-\mu)(y_t-\mu)^{\intercal}}\right| ^{\gamma}\notag\\
%&=o(1), %\label{eqn sufficient condition for Lindeberg condition II}
%\end{align}
%for some $\gamma>2$. The verification will be exactly the same as that of (\ref{eqn sufficient condition for Lindeberg condition}). In the end, we have
%\begin{align*}
%& T^{1-\frac{\gamma}{2}}(s^2n\kappa^2(W)\varpi)^{\gamma/2}\notag \\
%&  \cdot \mathbb{E} \left| c^{\intercal}(E^{\intercal}WE)^{-1}E^{\intercal}WD^+_nHP(D^{-1/2}\otimes D^{-1/2})\ve \sbr[1]{(y_t-\mu)(y_t-\mu)^{\intercal}-\mathbb{E}(y_t-\mu)(y_t-\mu)^{\intercal}}\right| ^{\gamma}\notag\\
%&=T^{1-\frac{\gamma}{2}}(s^2n\kappa^2(W)\varpi)^{\gamma/2}(\varpi \kappa(W) n)^{\gamma/2}O(\log^{\gamma} n)=O\del[3]{\frac{n^2\cdot \kappa^3(W)\cdot \log^4 n\cdot \varpi^2 }{T^{1-\frac{2}{\gamma}}}}^{\gamma/2} =o(1)
%\end{align*}
%by Assumption \ref{assu n indexed by T}(ii).
Verification of conditions (i)-(iii) of Theorem \ref{thm mcleish clt} in Appendix \ref{sec oldappendixB} will be exactly the same as those in Section \ref{sec asymptotic normality of t1 prime}, so we omit the details in the interest of space.

\subsubsection{$t_{1}-\hat{t}_{1}=o_p(1)$}

We now show that $t_{1}-\hat{t}_{1}=o_p(1)$. Let $A$ and $\hat{A}$ denote the numerators of $t_{1}$ and $\hat{t}_{1}$, respectively.
\begin{align*}
t_{1}-\hat{t}_{1} &=\frac{A}{\sqrt{c^{\intercal}Jc}}-\frac{\hat{A}}{\sqrt{c^{\intercal}\hat{J}_{T}c}}=\frac{\sqrt{s^2n\kappa^2(W)\varpi}A}{\sqrt{s^2n\kappa^2(W)\varpi c^{\intercal}Jc}}-\frac{\sqrt{s^2n\kappa^2(W)\varpi}\hat{A}}{\sqrt{s^2n\kappa^2(W)\varpi c^{\intercal}\hat{J}_{T}c}}.
\end{align*}
Since we have already shown in (\ref{eqn inverse square root G rate when D is unknown}) that  $s^2n\kappa^2(W)\varpi c^{\intercal}Jc$ is bounded away from zero by an absolute constant, it suffices to show the denominators as well as numerators of $t_{1}$ and $\hat{t}_{1}$ are asymptotically equivalent.

\subsubsection{Denominators of $t_{1}$ and $\hat{t}_{1}$}

We first show that the denominators of $t_{1}$ and $\hat{t}_{1}$ are asymptotically equivalent, i.e.,
\[s^2n\kappa^2(W)\varpi|c^{\intercal}\hat{J}_{T}c-c^{\intercal}Jc|=o_p(1).\]
 Define
\[c^{\intercal}\tilde{J}_Tc=c^{\intercal}(E^{\intercal}WE)^{-1}E^{\intercal}WD_{n}^{+}\hat
{H}_{T}\hat{P}_T(\hat{D}_T^{-1/2}\otimes \hat{D}_T^{-1/2})V(\hat{D}_T^{-1/2}\otimes \hat{D}_T^{-1/2})\hat{P}_T^{\intercal}\hat{H}_{T}D_{n}^{+^{\intercal}}WE(E^{\intercal
}WE)^{-1}c.\]
By the triangular inequality: $s^2n\kappa^2(W)\varpi|c^{\intercal}\hat{J}_{T}c-c^{\intercal}Jc|\leq s^2n\kappa^2(W)\varpi|c^{\intercal}\hat{J}_{T}c-c^{\intercal}\tilde{J}_Tc|+s^2n\kappa^2(W)\varpi|c^{\intercal}\tilde{J}_Tc-c^{\intercal}Jc|$. First, we prove $s^2n\kappa^2(W)\varpi|c^{\intercal}\hat{J}_{T}c-c^{\intercal}\tilde{J}_Tc|=o_p(1)$.
\begin{align*}
& s^2n\kappa^2(W)\varpi|c^{\intercal}\hat{J}_{T}c-c^{\intercal}\tilde{J}_Tc|\\
&=s^2n\kappa^2(W)\varpi|c^{\intercal}(E^{\intercal}WE)^{-1}E^{\intercal}WD_{n}^{+}\hat
{H}_{T}\hat{P}_T(\hat{D}_T^{-1/2}\otimes \hat{D}_T^{-1/2})\hat{V}_T(\hat{D}_T^{-1/2}\otimes \hat{D}_T^{-1/2})\hat{P}_T^{\intercal}\hat{H}_{T}D_{n}^{+^{\intercal}}WE(E^{\intercal
}WE)^{-1}c\\
&\qquad -c^{\intercal}(E^{\intercal}WE)^{-1}E^{\intercal}WD_{n}^{+}\hat
{H}_{T}\hat{P}_T(\hat{D}_T^{-1/2}\otimes \hat{D}_T^{-1/2})V(\hat{D}_T^{-1/2}\otimes \hat{D}_T^{-1/2})\hat{P}_T^{\intercal}\hat{H}_{T}D_{n}^{+^{\intercal}}WE(E^{\intercal
}WE)^{-1}c|\\
&=s^2n\kappa^2(W)\varpi\\
&\quad \cdot |c^{\intercal}(E^{\intercal}WE)^{-1}E^{\intercal}WD_{n}^{+}\hat
{H}_{T}\hat{P}_T(\hat{D}_T^{-1/2}\otimes \hat{D}_T^{-1/2})(\hat{V}_T-V)(\hat{D}_T^{-1/2}\otimes \hat{D}_T^{-1/2})\hat{P}_T^{\intercal}\hat{H}_{T}D_{n}^{+^{\intercal}}WE(E^{\intercal
}WE)^{-1}c|\\
&\leq s^2n\kappa^2(W)\varpi \|\hat{V}_T-V\|_{\infty}\|(\hat{D}_T^{-1/2}\otimes \hat{D}_T^{-1/2})\hat{P}_T^{\intercal}\hat{H}_{T}D_{n}^{+^{\intercal}}WE(E^{\intercal
}WE)^{-1}c\|_1^2\\
&\leq s^2n^3\kappa^2(W)\varpi \|\hat{V}_T-V\|_{\infty}\|(\hat{D}_T^{-1/2}\otimes \hat{D}_T^{-1/2})\hat{P}_T^{\intercal}\hat{H}_{T}D_{n}^{+^{\intercal}}WE(E^{\intercal
}WE)^{-1}c\|_2^2\\
&\leq s^2n^3\kappa^2(W)\varpi \|\hat{V}_T-V\|_{\infty}\|(\hat{D}_T^{-1/2}\otimes \hat{D}_T^{-1/2})\|_{\ell_2}^2\|\hat{P}_T^{\intercal}\|_{\ell_2}^2\|\hat{H}_{T}\|_{\ell_2}^2\|D_{n}^{+^{\intercal}}\|_{\ell_2}^2\|WE(E^{\intercal
}WE)^{-1}\|_{\ell_2}^2\\
&=O_p(s^2n^2\kappa^3(W)\varpi^2)\|\hat{V}_T-V\|_{\infty}=O_p\del [3]{\sqrt{\frac{n^4\kappa^6(W)s^4\varpi^4\log n}{T}}}=o_p(1),
\end{align*}
where $\|\cdot\|_{\infty}$ denotes the absolute elementwise maximum, the third equality is due to Lemma \ref{lemmaD}(v), Lemma \ref{lemma l2norm of kronecker product} in Appendix \ref{sec oldappendixB}, (\ref{eqn spectral norm of H and Hhat}), (\ref{align Eplus spectral norm}), (\ref{eqn spectral norm for Dnplus and Dn}) and (\ref{eqn spectral norm of P and Phat}), the second last equality is due to Lemma \ref{lemma rate for hatV-V infty} in SM \ref{sec A5}, and the last equality is due to Assumption \ref{assu n indexed by T}(ii).

We now prove $s^2n\kappa^2(W)\varpi|c^{\intercal}\tilde{J}_Tc-c^{\intercal}Jc|=o_p(1)$. Define
\begin{align*}
c^{\intercal}\tilde{J}_{T,a}c&:=c^{\intercal}(E^{\intercal}WE)^{-1}E^{\intercal}WD_{n}^{+}\hat
{H}_{T}\hat{P}_T(D^{-1/2}\otimes D^{-1/2})V(D^{-1/2}\otimes D^{-1/2})\hat{P}_T^{\intercal}\hat{H}_{T}D_{n}^{+^{\intercal}}WE(E^{\intercal
}WE)^{-1}c\\
c^{\intercal}\tilde{J}_{T,b}c&:=c^{\intercal}(E^{\intercal}WE)^{-1}E^{\intercal}WD_{n}^{+}\hat
{H}_{T}P(D^{-1/2}\otimes D^{-1/2})V(D^{-1/2}\otimes D^{-1/2})P^{\intercal}\hat{H}_{T}D_{n}^{+^{\intercal}}WE(E^{\intercal
}WE)^{-1}c.
\end{align*}
We use triangular inequality again
\begin{align}
& s^2n\kappa^2(W)\varpi|c^{\intercal}\tilde{J}_Tc-c^{\intercal}Jc|\leq\notag \\ 
& \quad s^2n\kappa^2(W)\varpi|c^{\intercal}\tilde{J}_Tc-c^{\intercal}\tilde{J}_{T,a}c|+s^2n\kappa^2(W)\varpi|c^{\intercal}\tilde{J}_{T,a}c-c^{\intercal}\tilde{J}_{T,b}c|+s^2n\kappa^2(W)\varpi|c^{\intercal}\tilde{J}_{T,b}c-c^{\intercal}Jc|. \label{eqn D unknown san jian ke}
\end{align}
We consider the first term on the right side of (\ref{eqn D unknown san jian ke}).
\begin{align}
& s^2n\kappa^2(W)\varpi|c^{\intercal}\tilde{J}_Tc-c^{\intercal}\tilde{J}_{T,a}c|= \notag\\
& s^2n\kappa^2(W)\varpi|c^{\intercal}(E^{\intercal}WE)^{-1}E^{\intercal}WD_{n}^{+}\hat
{H}_{T}\hat{P}_T(\hat{D}_T^{-1/2}\otimes \hat{D}_T^{-1/2})V(\hat{D}_T^{-1/2}\otimes \hat{D}_T^{-1/2})\hat{P}_T^{\intercal}\hat{H}_{T}D_{n}^{+^{\intercal}}WE(E^{\intercal
}WE)^{-1}c\notag \\
&\qquad -c^{\intercal}(E^{\intercal}WE)^{-1}E^{\intercal}WD_{n}^{+}\hat
{H}_{T}\hat{P}_T(D^{-1/2}\otimes D^{-1/2})V(D^{-1/2}\otimes D^{-1/2})\hat{P}_T^{\intercal}\hat{H}_{T}D_{n}^{+^{\intercal}}WE(E^{\intercal
}WE)^{-1}c| \notag\\
&\leq s^2n\kappa^2(W)\varpi \envert[1]{\text{maxeval}(V)}\|(\hat{D}_T^{-1/2}\otimes \hat{D}_T^{-1/2}-D^{-1/2}\otimes D^{-1/2})\hat{P}_T^{\intercal}\hat{H}_{T}D_{n}^{+^{\intercal}}WE(E^{\intercal}WE)^{-1}c\|_2^2\notag \\
&\quad+s^2n\kappa^2(W)\varpi\|V(D^{-1/2} \otimes D^{-1/2} )\hat{P}_T^{\intercal}\hat{H}_{T}D_{n}^{+^{\intercal}}WE(E^{\intercal}WE)^{-1}c\|_2\notag\\
&\qquad \cdot\|(\hat{D}_T^{-1/2}\otimes \hat{D}_T^{-1/2}-D^{-1/2}\otimes D^{-1/2})\hat{P}_T^{\intercal}\hat{H}_{T}D_{n}^{+^{\intercal}}WE(E^{\intercal}WE)^{-1}c\|_2 \label{align tildeG-G D unknown I}
\end{align}
where the inequality is due to Lemma \ref{lemma vandegeer} in Appendix \ref{sec oldappendixB}. We consider the first term of (\ref{align tildeG-G D unknown I}) first.
\begin{align*}
& s^2n\kappa^2(W)\varpi\envert[1]{\text{maxeval}(V)}\|(\hat{D}_T^{-1/2}\otimes \hat{D}_T^{-1/2}-D^{-1/2}\otimes D^{-1/2})\hat{P}_T^{\intercal}\hat{H}_{T}D_{n}^{+^{\intercal}}WE(E^{\intercal}WE)^{-1}c\|_2^2\\
& =O(s^2n\kappa^2(W)\varpi)\|\hat{D}_T^{-1/2}\otimes \hat{D}_T^{-1/2}-D^{-1/2}\otimes D^{-1/2}\|_{\ell_2}^2\|\hat{P}_T^{\intercal}\|_{\ell_2}^2\|\hat{H}_T\|_{\ell_2}^2\|D_{n}^{+^{\intercal}}\|_{\ell_2}^2\|WE(E^{\intercal}WE)^{-1}\|^2_{\ell_2}\\
&= O_p(s^2n\kappa^3(W)\varpi^2/T)=o_p(1),
\end{align*}
where the second last equality is due to (\ref{eqn spectral norm of H and Hhat}), (\ref{eqn spectral norm for Dnplus and Dn}), (\ref{align Eplus spectral norm}), (\ref{eqn spectral norm of P and Phat}) and Lemma \ref{lemmaD}(vii), and the last equality is due to Assumption \ref{assu n indexed by T}(ii).

We now consider the second term of (\ref{align tildeG-G D unknown I}).
\begin{align*}
& 2 s^2n\kappa^2(W)\varpi\|V(D^{-1/2} \otimes D^{-1/2} )\hat{P}_T^{\intercal}\hat{H}_{T}D_{n}^{+^{\intercal}}WE(E^{\intercal}WE)^{-1}c\|_2\notag\\
&\qquad \cdot\|(\hat{D}_T^{-1/2}\otimes \hat{D}_T^{-1/2}-D^{-1/2}\otimes D^{-1/2})\hat{P}_T^{\intercal}\hat{H}_{T}D_{n}^{+^{\intercal}}WE(E^{\intercal}WE)^{-1}c\|_2\\
&\leq O(s^2n\kappa^2(W)\varpi)\|\hat{D}_T^{-1/2}\otimes \hat{D}_T^{-1/2}-D^{-1/2}\otimes D^{-1/2}\|_{\ell_2}\|\hat{P}_T^{\intercal}\|_{\ell_2}^2\|\hat{H}_T\|_{\ell_2}^2\|D_{n}^{+^{\intercal}}\|_{\ell_2}^2\|WE(E^{\intercal}WE)^{-1}\|^2_{\ell_2}\\
&=O(\sqrt{s^4n\kappa^6(W)\varpi^4/T})=o_p(1),
\end{align*}
where the first equality is due to (\ref{eqn spectral norm of H and Hhat}), (\ref{eqn spectral norm for Dnplus and Dn}), (\ref{align Eplus spectral norm}), (\ref{eqn spectral norm of P and Phat}) and Lemma \ref{lemmaD}(vii), and the last equality is due to Assumption \ref{assu n indexed by T}(ii). We have proved $s^2n\kappa^2(W)\varpi|c^{\intercal}\tilde{J}_Tc-c^{\intercal}\tilde{J}_{T,a}c|=o_p(1)$.

% san jian ke II
We consider the second term on the right hand side of (\ref{eqn D unknown san jian ke}).
\begin{align}
& s^2n\kappa^2(W)\varpi|c^{\intercal}\tilde{J}_{T,a}c-c^{\intercal}\tilde{J}_{T,b}c| =\notag\\
& s^2n\kappa^2(W)\varpi|c^{\intercal}(E^{\intercal}WE)^{-1}E^{\intercal}WD_{n}^{+}\hat
{H}_{T}\hat{P}_T(D^{-1/2}\otimes D^{-1/2})V(D^{-1/2}\otimes D^{-1/2})\hat{P}_T^{\intercal}\hat{H}_{T}D_{n}^{+^{\intercal}}WE(E^{\intercal
}WE)^{-1}c\notag \\
&\qquad -c^{\intercal}(E^{\intercal}WE)^{-1}E^{\intercal}WD_{n}^{+}\hat
{H}_{T}P(D^{-1/2}\otimes D^{-1/2})V(D^{-1/2}\otimes D^{-1/2})P^{\intercal}\hat{H}_{T}D_{n}^{+^{\intercal}}WE(E^{\intercal
}WE)^{-1}c| \notag\\
&\leq s^2n\kappa^2(W)\varpi \envert[1]{\text{maxeval}[(D^{-1/2}\otimes D^{-1/2})V(D^{-1/2}\otimes D^{-1/2})]}\|(\hat{P}_T-P)^{\intercal}\hat{H}_{T}D_{n}^{+^{\intercal}}WE(E^{\intercal}WE)^{-1}c\|_2^2\notag \\
&\quad+2 s^2n\kappa^2(W)\varpi\|(D^{-1/2} \otimes D^{-1/2} )V(D^{-1/2} \otimes D^{-1/2} )P^{\intercal}\hat{H}_{T}D_{n}^{+^{\intercal}}WE(E^{\intercal}WE)^{-1}c\|_2\notag\\
&\qquad \cdot\|(\hat{P}_T-P)^{\intercal}\hat{H}_{T}D_{n}^{+^{\intercal}}WE(E^{\intercal}WE)^{-1}c\|_2 \label{align tildeG-G D unknown II}
\end{align}
where the inequality is due to Lemma \ref{lemma vandegeer} in Appendix \ref{sec oldappendixB}. We consider the first term of (\ref{align tildeG-G D unknown II}) first.
\begin{align*}
& s^2n\kappa^2(W)\varpi \envert[1]{\text{maxeval}[(D^{-1/2}\otimes D^{-1/2})V(D^{-1/2}\otimes D^{-1/2})]}\|(\hat{P}_T-P)^{\intercal}\hat{H}_{T}D_{n}^{+^{\intercal}}WE(E^{\intercal}WE)^{-1}c\|_2^2\\
& =O(s^2n\kappa^2(W)\varpi)\|\hat{P}_T^{\intercal}-P^{\intercal}\|_{\ell_2}^2\|\hat{H}_T\|_{\ell_2}^2\|D_{n}^{+^{\intercal}}\|_{\ell_2}^2\|WE(E^{\intercal}WE)^{-1}\|^2_{\ell_2}\\
&= O_p(s^2n\kappa^3(W)\varpi^2/T)=o_p(1),
\end{align*}
where the second last equality is due to (\ref{eqn spectral norm of H and Hhat}), (\ref{eqn spectral norm for Dnplus and Dn}), (\ref{align Eplus spectral norm}), and (\ref{eqn spectral norm of P and Phat}), and the last equality is due to Assumption \ref{assu n indexed by T}(ii).

We now consider the second term of (\ref{align tildeG-G D unknown II}).
\begin{align*}
& 2 s^2n\kappa^2(W)\varpi\|(D^{-1/2} \otimes D^{-1/2} )V(D^{-1/2} \otimes D^{-1/2} )P^{\intercal}\hat{H}_{T}D_{n}^{+^{\intercal}}WE(E^{\intercal}WE)^{-1}c\|_2\notag\\
&\qquad \cdot\|(\hat{P}_T-P)^{\intercal}\hat{H}_{T}D_{n}^{+^{\intercal}}WE(E^{\intercal}WE)^{-1}c\|_2\\
&\leq O(s^2n\kappa^2(W)\varpi)\|\hat{P}_T^{\intercal}-P^{\intercal}\|_{\ell_2}^2\|\hat{H}_T\|_{\ell_2}^2\|D_{n}^{+^{\intercal}}\|_{\ell_2}^2\|WE(E^{\intercal}WE)^{-1}\|^2_{\ell_2}\\
&=O(\sqrt{s^4n\kappa^6(W)\varpi^4/T})=o_p(1),
\end{align*}
where the first equality is due to (\ref{eqn spectral norm of H and Hhat}), (\ref{eqn spectral norm for Dnplus and Dn}), (\ref{align Eplus spectral norm}), and (\ref{eqn spectral norm of P and Phat}), and the last equality is due to Assumption \ref{assu n indexed by T}(ii). We have proved $s^2n\kappa^2(W)\varpi|c^{\intercal}\tilde{J}_{T,a}c-c^{\intercal}\tilde{J}_{T,b}c|=o_p(1)$.

% san jian ke III
We consider the third term on the right hand side of (\ref{eqn D unknown san jian ke}).
\begin{align}
& s^2n\kappa^2(W)\varpi|c^{\intercal}\tilde{J}_{T,b}c-c^{\intercal}Jc| =\notag\\
& s^2n\kappa^2(W)\varpi|c^{\intercal}(E^{\intercal}WE)^{-1}E^{\intercal}WD_{n}^{+}\hat
{H}_{T}P(D^{-1/2}\otimes D^{-1/2})V(D^{-1/2}\otimes D^{-1/2})P^{\intercal}\hat{H}_{T}D_{n}^{+^{\intercal}}WE(E^{\intercal
}WE)^{-1}c\notag \\
&\qquad -c^{\intercal}(E^{\intercal}WE)^{-1}E^{\intercal}WD_{n}^{+}HTP(D^{-1/2}\otimes D^{-1/2})V(D^{-1/2}\otimes D^{-1/2})P^{\intercal}HD_{n}^{+^{\intercal}}WE(E^{\intercal
}WE)^{-1}c| \notag\\
&\leq s^2n\kappa^2(W)\varpi \envert[1]{\text{maxeval}[P(D^{-1/2}\otimes D^{-1/2})V(D^{-1/2}\otimes D^{-1/2})P^{\intercal}]}\|(\hat{H}_{T}-H)D_{n}^{+^{\intercal}}WE(E^{\intercal}WE)^{-1}c\|_2^2\notag \\
&\quad+2 s^2n\kappa^2(W)\varpi\|P(D^{-1/2} \otimes D^{-1/2} )V(D^{-1/2} \otimes D^{-1/2} )P^{\intercal}HD_{n}^{+^{\intercal}}WE(E^{\intercal}WE)^{-1}c\|_2\notag\\
&\qquad \cdot\|(\hat{H}_{T}-H)D_{n}^{+^{\intercal}}WE(E^{\intercal}WE)^{-1}c\|_2 \label{align tildeG-G D unknown III}
\end{align}
where the inequality is due to Lemma \ref{lemma vandegeer} in Appendix \ref{sec oldappendixB}. We consider the first term of (\ref{align tildeG-G D unknown III}) first.
\begin{align*}
& s^2n\kappa^2(W)\varpi \envert[1]{\text{maxeval}[P(D^{-1/2}\otimes D^{-1/2})V(D^{-1/2}\otimes D^{-1/2})P^{\intercal}]}\|(\hat{H}_{T}-H)D_{n}^{+^{\intercal}}WE(E^{\intercal}WE)^{-1}c\|_2^2\\
& =O(s^2n\kappa^2(W)\varpi)\|\hat{H}_T-H\|_{\ell_2}^2\|D_{n}^{+^{\intercal}}\|_{\ell_2}^2\|WE(E^{\intercal}WE)^{-1}\|^2_{\ell_2}\\
&= O_p(s^2n\kappa^3(W)\varpi^2/T)=o_p(1),
\end{align*}
where the second last equality is due to (\ref{eqn spectral norm of H and Hhat}), (\ref{eqn spectral norm for Dnplus and Dn}), and (\ref{align Eplus spectral norm}), and the last equality is due to Assumption \ref{assu n indexed by T}(ii).

We now consider the second term of (\ref{align tildeG-G D unknown III}).
\begin{align*}
& 2 s^2n\kappa^2(W)\varpi\|P(D^{-1/2} \otimes D^{-1/2} )V(D^{-1/2} \otimes D^{-1/2} )P^{\intercal}HD_{n}^{+^{\intercal}}WE(E^{\intercal}WE)^{-1}c\|_2\notag\\
&\qquad \cdot\|(\hat{H}_{T}-H)D_{n}^{+^{\intercal}}WE(E^{\intercal}WE)^{-1}c\|_2\\
&\leq O(s^2n\kappa^2(W)\varpi)\|\hat{H}_T-H\|_{\ell_2}^2\|D_{n}^{+^{\intercal}}\|_{\ell_2}^2\|WE(E^{\intercal}WE)^{-1}\|^2_{\ell_2}=O(\sqrt{s^4n\kappa^6(W)\varpi^4/T})=o_p(1),
\end{align*}
where the first equality is due to (\ref{eqn spectral norm of H and Hhat}), (\ref{eqn spectral norm for Dnplus and Dn}), and (\ref{align Eplus spectral norm}), and the last equality is due to Assumption \ref{assu n indexed by T}(ii). We have proved $s^2n\kappa^2(W)\varpi|c^{\intercal}\tilde{J}_{T,b}c-c^{\intercal}Jc|=o_p(1)$. Hence we have proved $s^2n\kappa^2(W)\varpi|c^{\intercal}\tilde{J}_Tc-c^{\intercal}Jc|=o_p(1)$.

\subsubsection{Numerators of $t_{1}$ and $\hat{t}_{1}$}

We now show that numerators of $t_{1}$ and $\hat{t}_{1}$ are asymptotically equivalent, i.e.,
\[\sqrt{s^2n\kappa^2(W)\varpi}|A-\hat{A}|=o_p(1).\]
Note that
\begin{align*}
 \hat{A}&=\sqrt{T}c^{\intercal}(E^{\intercal}WE)^{-1}E^{\intercal}WD_n^+H\left. \frac{\partial \ve \Theta}{\partial \ve \Sigma}\right| _{\Sigma=\mathring{\Sigma}_T^{(i)}}\ve (\hat{\Sigma}_T-\Sigma)\\
&=\sqrt{T}c^{\intercal}(E^{\intercal}WE)^{-1}E^{\intercal}WD_n^+H\left. \frac{\partial \ve \Theta}{\partial \ve \Sigma}\right| _{\Sigma=\mathring{\Sigma}_T^{(i)}}\ve (\hat{\Sigma}_T-\tilde{\Sigma}_T)\\
&\qquad +\sqrt{T}c^{\intercal}(E^{\intercal}WE)^{-1}E^{\intercal}WD_n^+H\left. \frac{\partial \ve \Theta}{\partial \ve \Sigma}\right| _{\Sigma=\mathring{\Sigma}_T^{(i)}}\ve (\tilde{\Sigma}_T-\Sigma)\\
&=:\hat{A}_a+\hat{A}_b.
\end{align*}
To show $\sqrt{s^2n\kappa^2(W)\varpi}|A-\hat{A}|=o_p(1)$, it suffices to show $\sqrt{s^2n\kappa^2(W)\varpi}|\hat{A}_b-A|=o_p(1)$ and $\sqrt{s^2n\kappa^2(W)\varpi}|\hat{A}_a|=o_p(1)$. We first show that $\sqrt{s^2n\kappa^2(W)\varpi}|\hat{A}_b-A|=o_p(1)$.
\begin{align*}
&\sqrt{s^2n\kappa^2(W)\varpi}|\hat{A}_b-A|\\
&=\sqrt{s^2n\kappa^2(W)\varpi}\envert[3]{\sqrt{T}c^{\intercal}(E^{\intercal}WE)^{-1}E^{\intercal}WD_n^+H\sbr[3]{\left. \frac{\partial \ve \Theta}{\partial \ve \Sigma}\right| _{\Sigma=\mathring{\Sigma}_T^{(i)}}-P(D^{-1/2}\otimes D^{-1/2})}\ve (\tilde{\Sigma}_T-\Sigma)}\\
&\leq \sqrt{Ts^2n\kappa^2(W)\varpi}\|(E^{\intercal}WE)^{-1}E^{\intercal}W\|_{\ell_2}\|D_n^+\|_{\ell_2}\|H\|_{\ell_2}\\
&\qquad \cdot \enVert[3]{\left. \frac{\partial \ve \Theta}{\partial \ve \Sigma}\right| _{\Sigma=\mathring{\Sigma}_T^{(i)}}-P(D^{-1/2}\otimes D^{-1/2})}_{\ell_2}\|\ve (\tilde{\Sigma}_T-\Sigma)\|_{2}\\
& =O(\sqrt{Ts^2n\kappa^2(W)\varpi})\sqrt{\varpi \kappa(W)/n}O_p\del [3]{\sqrt{\frac{n}{T}}}\|\tilde{\Sigma}_T-\Sigma\|_{F}\leq O(\sqrt{n s^2\kappa^3(W)\varpi^2})\sqrt{n}\|\tilde{\Sigma}_T-\Sigma\|_{\ell_2}\\
&=O(\sqrt{n s^2\kappa^3(W)\varpi^2})\sqrt{n}O_p\del [3]{\sqrt{\frac{n}{T}}}=O_p\del [3]{\sqrt{\frac{n^3s^2\kappa^3(W)\varpi^2}{T}}}=o_p(1),
\end{align*}
where the second equality is due to Assumption \ref{assu when D is unknown}(i), the third equality is due to Lemma \ref{lemma spectral norm perturbation covariance matrix}, and final equality is due to Assumption \ref{assu n indexed by T}(ii).

We now show that $\sqrt{s^2n\kappa^2(W)\varpi}|\hat{A}_a|=o_p(1)$.
\begin{align*}
&\sqrt{s^2n\kappa^2(W)\varpi T}\envert[3]{c^{\intercal}(E^{\intercal}WE)^{-1}E^{\intercal}WD_n^+H\left. \frac{\partial \ve \Theta}{\partial \ve \Sigma}\right| _{\Sigma=\mathring{\Sigma}_T^{(i)}}\ve (\hat{\Sigma}_T-\tilde{\Sigma}_T)}\\
&=\sqrt{s^2n\kappa^2(W)\varpi T}\envert[3]{c^{\intercal}(E^{\intercal}WE)^{-1}E^{\intercal}WD_n^+H\left. \frac{\partial \ve \Theta}{\partial \ve \Sigma}\right| _{\Sigma=\mathring{\Sigma}_T^{(i)}}\ve \sbr[1]{(\bar{y}-\mu)(\bar{y}-\mu)^{\intercal}}}\\
&\leq \sqrt{s^2n\kappa^2(W)\varpi T}\|(E^{\intercal}WE)^{-1}E^{\intercal}W\|_{\ell_2}\|D_n^+\|_{\ell_2}\|H\|_{\ell_2}\enVert[3]{\left. \frac{\partial \ve \Theta}{\partial \ve \Sigma}\right| _{\Sigma=\mathring{\Sigma}_T^{(i)}}}_{\ell_2}\|\ve \sbr[1]{(\bar{y}-\mu)(\bar{y}-\mu)^{\intercal}}\|_{2}\\
&=O(\sqrt{Ts^2n\kappa^2(W)\varpi})\sqrt{\varpi \kappa(W)/n}\|(\bar{y}-\mu)(\bar{y}-\mu)^{\intercal}\|_F\\
&\leq O(\sqrt{Ts^2n\kappa^2(W)\varpi})\sqrt{\varpi \kappa(W)/n}n\|(\bar{y}-\mu)(\bar{y}-\mu)^{\intercal}\|_{\infty}\\
&=O(\sqrt{Ts^2n^2\kappa^3(W)\varpi^2})\max_{1\leq i,j\leq n}\envert[1]{(\bar{y}-\mu)_i(\bar{y}-\mu)_j}=O_p(\sqrt{Ts^2n^2\kappa^3(W)\varpi^2})\log n/T\\
&=O_p\del[3]{\sqrt{\frac{\log^4n\cdot n^2\kappa^3(W)\varpi^2}{T}}}=o_p(1),
\end{align*}
where the third last equality is due to (\ref{eqn rate for hatV-V infinity part D}), the last equality is due to Assumption \ref{assu n indexed by T}(ii), and the second equality is due to (\ref{eqn spectral norm of H and Hhat}), (\ref{eqn spectral norm for Dnplus and Dn}), (\ref{align Eplus spectral norm}), and the fact that
\begin{align*}
\enVert[3]{\left. \frac{\partial \ve \Theta}{\partial \ve \Sigma}\right| _{\Sigma=\mathring{\Sigma}_T^{(i)}}}_{\ell_2}&=\enVert[3]{\left. \frac{\partial \ve \Theta}{\partial \ve \Sigma}\right| _{\Sigma=\mathring{\Sigma}_T^{(i)}}-P(D^{-1/2}\otimes D^{-1/2})}_{\ell_2}+\enVert[3]{P(D^{-1/2}\otimes D^{-1/2})}_{\ell_2}\\
&=O_p\del[3]{\sqrt{\frac{n}{T}}}+O(1)=O_p(1).
\end{align*}

\subsubsection{$\hat{t}_{2}=o_p(1)$}

Write
\[\hat{t}_{2}=\frac{\sqrt{T}\sqrt{s^2n\kappa^2(W)\varpi}c^{\intercal}(E^{\intercal}WE)^{-1}E^{\intercal}WD_n^+\ve O_p(\|\hat{\Theta}_{T}-\Theta\|_{\ell_2}^2)}{\sqrt{s^2n\kappa^2(W)\varpi c^{\intercal}\hat{J}_{T}c}}.\]
Since the denominator of the preceding equation is bounded away from zero by an absolute constant with probability approaching one by (\ref{eqn inverse square root G rate when D is unknown}) and that $s^2n\kappa^2(W)\varpi|c^{\intercal}\hat{J}_{T}c-c^{\intercal}Jc|=o_p(1)$, it suffices to show
\[\sqrt{T}\sqrt{s^2n\kappa^2(W)\varpi}c^{\intercal}(E^{\intercal}WE)^{-1}E^{\intercal}WD_n^+\ve O_p(\|\hat{\Theta}_{T}-\Theta\|_{\ell_2}^2)=o_p(1).\]
This is straightforward:
\begin{align*}
&|\sqrt{Ts^2n\kappa^2(W)\varpi}c^{\intercal}(E^{\intercal}WE)^{-1}E^{\intercal}WD^+_n\ve O_p(\|\hat{\Theta}_{T}-\Theta\|_{\ell_2}^2)|\\
&\leq \sqrt{Ts^2n\kappa^2(W)\varpi}\|c^{\intercal}(E^{\intercal}WE)^{-1}E^{\intercal}WD^+_n\|_2\|\ve O_p(\|\hat{\Theta}_{T}-\Theta\|_{\ell_2}^2)\|_2\\
&= O(\sqrt{Ts^2\kappa^3(W)\varpi^2})\|O_p(\|\hat{\Theta}_{T}-\Theta\|_{\ell_2}^2)\|_F= O(\sqrt{Tns^2\kappa^3(W)\varpi^2})\|O_p(\|\hat{\Theta}_{T}-\Theta\|_{\ell_2}^2)\|_{\ell_2}\\
&=O(\sqrt{Tns^2\kappa^3(W)\varpi^2})O_p(\|\hat{\Theta}_{T}-\Theta\|_{\ell_2}^2) =O_p\del[3]{ \sqrt{\frac{n^3s^2\kappa^3(W)\varpi^2}{T}}}=o_p(1),
\end{align*}
where the last equality is due to Assumption \ref{assu n indexed by T}(ii).
\end{proof}

\subsection{Proof of Theorem \ref{prop Haihan score functions and second derivatives}}
\label{sec A7}

In this subsection, we give a proof for Theorem \ref{prop Haihan score functions and second derivatives}. We first give a useful lemma which is used in the proof of Theorem \ref{prop Haihan score functions and second derivatives}.

\begin{lemma}[\cite{magnusneudecker2007} p218]
\label{thmmagnusneudecker2007p218}
Let $\phi$ be a twice differentiable real-valued function of an $n\times q$ matrix $X$. Then the following two relationships hold between the second differential and the Hessian matrix of $\phi$ at $X$:
\[d^2\phi(X)=\text{tr}\sbr[1]{B(dX)^{\intercal}CdX}\quad \Longleftrightarrow \quad \frac{\partial ^2\phi(X)}{\partial (\ve X)\partial (\ve X)^{\intercal}}=\frac{1}{2}(B^{\intercal}\otimes C+B\otimes C^{\intercal})\]
and
\[d^2\phi(X)=\text{tr}\sbr[1]{B(dX)CdX}\quad \Longleftrightarrow \quad \frac{\partial ^2\phi(X)}{\partial (\ve X)\partial (\ve X)^{\intercal}}=\frac{1}{2}K_{qn}(B^{\intercal}\otimes C+C^{\intercal}\otimes B).\]
\end{lemma}

\bigskip 

We are now ready to give the proof of Theorem \ref{prop Haihan score functions and second derivatives}.

\begin{proof}[Proof of Theorem \ref{prop Haihan score functions and second derivatives}]
For part (i), letting $A$ denote $D^{-1/2}\tilde{\Sigma}_T D^{-1/2}$, we take the first differential of $\ell_{T,D}(\theta,\mu)$ with respect to $\Omega(\theta)$:
\begin{align}
& d \ell_{T,D}(\theta,\mu) = -\frac{T}{2}d\log \envert[1]{e^{\Omega}}-\frac{1}{2}d\sum_{t=1}^{T}\text{tr} \sbr[2]{(y_t-\mu)^{\intercal} D^{-1/2} e^{-\Omega}D^{-1/2}(y_t-\mu)}\notag \\
&=-\frac{T}{2}d\log \envert[1]{e^{\Omega}}-\frac{T}{2}d\text{tr} \sbr[3]{D^{-1/2}\frac{1}{T}\sum_{t=1}^{T}(y_t-\mu)(y_t-\mu)^{\intercal} D^{-1/2} e^{-\Omega}}\notag\\
&=-\frac{T}{2}d\log \envert[1]{e^{\Omega}}-\frac{T}{2}d\text{tr} \sbr[1]{A e^{-\Omega}}=-\frac{T}{2}\text{tr}(e^{-\Omega}de^{\Omega})-\frac{T}{2}\text{tr} \del[1]{A de^{-\Omega}}\notag\\
&=-\frac{T}{2}\text{tr}(e^{-\Omega}de^{\Omega})+\frac{T}{2}\text{tr} \del[1]{A e^{-\Omega}(de^{\Omega})e^{-\Omega}}\notag\\
&=-\frac{T}{2}\text{tr}(e^{-\Omega}de^{\Omega})+\frac{T}{2}\text{tr} \del[1]{e^{-\Omega}A e^{-\Omega}de^{\Omega}}\label{align first differential}\\
& = \frac{T}{2}\text{tr} \sbr[2]{\del[1]{e^{-\Omega}A e^{-\Omega}-e^{-\Omega}}de^{\Omega}}=\frac{T}{2}\del[3]{\ve\sbr[2]{\del[1]{e^{-\Omega}A e^{-\Omega}-e^{-\Omega}}^{\intercal}}}^{\intercal}\ve de^{\Omega}\notag\\
&=\frac{T}{2}\del[2]{\ve\sbr[1]{e^{-\Omega}A e^{-\Omega}-e^{-\Omega}}}^{\intercal}\ve \sbr[3]{\int_{0}^{1}e^{(1-t)\Omega}(d\Omega) e^{t\Omega}dt}\notag\\
&=\frac{T}{2}\del[2]{\ve\sbr[1]{e^{-\Omega}A e^{-\Omega}-e^{-\Omega}}}^{\intercal} \sbr[3]{\int_{0}^{1}e^{t\Omega}\otimes e^{(1-t)\Omega}dt}d\ve\Omega\notag\\
&=\frac{T}{2}\del[2]{\ve\sbr[1]{e^{-\Omega}A e^{-\Omega}-e^{-\Omega}}}^{\intercal} \sbr[3]{\int_{0}^{1}e^{t\Omega}\otimes e^{(1-t)\Omega}dt}D_n Ed\theta\notag
\end{align}
where the fourth equality is due to that $d\log |X|=\text{tr}(X^{-1}dX)$ for any square matrix $X$, the fifth equality is due to that $dX^{-1}=-X^{-1}(dX)X^{-1}$, the six equality is due to the cyclic property of trace operator, the eighth equality is due to that $\text{tr}(AB)=(\ve [A^{\intercal}])^{\intercal}\ve B$, the ninth equality is due to that $de^{\Omega}=\int_{0}^{1}e^{(1-t)\Omega}(d\Omega) e^{t\Omega}dt$ (c.f. (10.15) in \cite{higham2008} p238), the second last equality is due to that $\ve (ABC)=(C^{\intercal}\otimes A)\ve B$, and the last equality is due to $\ve \Omega=D_n\vech \Omega=D_nE\theta$. Thus, we conclude that
\[\frac{\partial \ell_{T,D}(\theta,\mu)}{\partial \theta^{\intercal}}=\frac{T}{2}E^{\intercal}D_n^{\intercal}\sbr[3]{\int_{0}^{1}e^{t\Omega}\otimes e^{(1-t)\Omega}dt}\ve\sbr[1]{e^{-\Omega}D^{-1/2}\tilde{\Sigma}_T D^{-1/2} e^{-\Omega}-e^{-\Omega}} .\]

\bigskip

For part (ii), the $s\times s$ block of the Hessian matrix of (\ref{align likelihood fun when D known}) corresponding to $\theta$ is more difficult to derive. There are two approaches; they give the same Hessian but sometimes it is difficult to see the equivalence because of the presence of Kronecker products, duplication matrices etc. The first approach is to differentiate the score function with respect to $\theta$ again. The second approach is to start from (\ref{align first differential}), take differential again, manipulate the final result into the canonical form, and extract the Hessian from the canonical form. The second approach is due to \cite{magnusneudecker2007}; \cite{minka2000} provided an easily accessible introduction to this approach. We shall use the second approach to derive the Hessian matrix.

There are two terms in (\ref{align first differential}). The first term could be simplified into
\begin{align*}
&-\frac{T}{2}\text{tr}(e^{-\Omega}de^{\Omega})=-\frac{T}{2}\text{tr}\del[3]{e^{-\Omega}\int_{0}^{1}e^{(1-t)\Omega}(d\Omega) e^{t\Omega}dt}=-\frac{T}{2}\int_{0}^{1}\text{tr}\del[1]{e^{-\Omega}e^{(1-t)\Omega}(d\Omega) e^{t\Omega}}dt\\
&=-\frac{T}{2}\int_{0}^{1}\text{tr}\del[1]{e^{-t\Omega}(d\Omega) e^{t\Omega}}dt=-\frac{T}{2}\int_{0}^{1}\text{tr}\del[1]{d\Omega }dt=-\frac{T}{2}\text{tr}\del[1]{d\Omega }
\end{align*}
whence we see that it is not a function of $\Omega$ ($d\Omega$ is not a function of $\Omega$). Thus taking differential of (\ref{align first differential}) will cause this term drop out. We now take the differential of the second term in (\ref{align first differential}):
\begin{align*}
& d \frac{T}{2}\text{tr} \del[1]{e^{-\Omega}A e^{-\Omega}de^{\Omega}}= d \frac{T}{2}\text{tr} \del[3]{e^{-\Omega}A e^{-\Omega}\int_{0}^{1}e^{(1-t)\Omega}(d\Omega) e^{t\Omega}dt}\\
&= d\frac{T}{2} \int_{0}^{1}\text{tr} \del[2]{e^{(t-1)\Omega}A e^{-t\Omega}d\Omega }dt= \frac{T}{2} \int_{0}^{1}\text{tr} \del[2]{(de^{(t-1)\Omega})A e^{-t\Omega}d\Omega +e^{(t-1)\Omega}A (de^{-t\Omega})d\Omega}dt\\
&=\frac{T}{2} \int_{0}^{1}\text{tr} \del[2]{\int_{0}^{1}e^{(1-s)(t-1)\Omega}(d(t-1)\Omega)e^{s(t-1)\Omega}dsA e^{-t\Omega}d\Omega}dt\\ &\qquad+\frac{T}{2} \int_{0}^{1}\text{tr} \del[2]{e^{(t-1)\Omega}A \int_{0}^{1}e^{-(1-s)t\Omega}(d(-t)\Omega)e^{-st\Omega}dsd\Omega}dt\\
&=-\frac{T}{2} \int_{0}^{1}\int_{0}^{1}\text{tr} \del[2]{e^{-(1-s)(1-t)\Omega}(d\Omega)e^{-s(1-t)\Omega}A e^{-t\Omega}d\Omega}ds\cdot (1-t)dt
\\ &\qquad-\frac{T}{2} \int_{0}^{1}\int_{0}^{1}\text{tr} \del[2]{e^{-(1-t)\Omega}A e^{-(1-s)t\Omega}(d\Omega)e^{-st\Omega}d\Omega}ds\cdot tdt.
\end{align*}
We next invoke Lemma \ref{thmmagnusneudecker2007p218} to get
\begin{align*}
&\frac{\partial^2 \ell_{T,D}(\theta,\mu)}{\partial \ve\Omega \partial (\ve \Omega)^{\intercal}}=\\
&-\frac{T}{2} \int_{0}^{1}\int_{0}^{1}\frac{1}{2}K_{n,n}\del[2]{ e^{-(1-s)(1-t)\Omega}\otimes e^{-s(1-t)\Omega}A e^{-t\Omega}+e^{-t\Omega}Ae^{-s(1-t)\Omega}\otimes e^{-(1-s)(1-t)\Omega}}ds\cdot (1-t)dt
\\ &\qquad-\frac{T}{2} \int_{0}^{1}\int_{0}^{1}\frac{1}{2}K_{n,n} \del[2]{e^{-(1-s)t\Omega}Ae^{-(1-t)\Omega}\otimes e^{-st\Omega}+e^{-st\Omega}\otimes e^{-(1-t)\Omega}A e^{-(1-s)t\Omega} }ds\cdot tdt\\
&=-\frac{T}{2} \int_{0}^{1}\int_{0}^{1}\frac{1}{2}K_{n,n}\del[2]{ e^{-st\Omega}\otimes e^{-(1-s)t\Omega}A e^{-(1-t)\Omega}+e^{-(1-t)\Omega}Ae^{-(1-s)t\Omega}\otimes e^{-st\Omega}}ds\cdot tdt
\\ &\qquad-\frac{T}{2} \int_{0}^{1}\int_{0}^{1}\frac{1}{2}K_{n,n} \del[2]{e^{-(1-s)t\Omega}Ae^{-(1-t)\Omega}\otimes e^{-st\Omega}+e^{-st\Omega}\otimes e^{-(1-t)\Omega}A e^{-(1-s)t\Omega} }ds\cdot tdt
\end{align*}
where the second equality is due to change of variables $1-t\mapsto t$ and $1-s\mapsto s$ for the first term only. Note that although we have used symmetry of $\Omega$ throughout the derivation, we have not yet incorporated this fact into the Hessian. In our case, there is no need to incorporate symmetry of $\Omega$ into the Hessian because our ultimate goal is to get the Hessian in terms of the unique elements of $\Omega$, $\theta$ (see \cite{minka2000} for more explanations of this). Thus the final Hessian in terms of $\theta$ is
\begin{align*}
&\frac{\partial^2 \ell_{T,D}(\theta,\mu)}{\partial \theta \partial \theta^{\intercal}}=\\
&  -\frac{T}{2} \int_{0}^{1}\int_{0}^{1}\frac{1}{2}E^{\intercal}D_n^{\intercal}K_{n,n}\del[1]{ e^{-st\Omega}\otimes e^{-(1-s)t\Omega}A e^{-(1-t)\Omega}+e^{-(1-t)\Omega}Ae^{-(1-s)t\Omega}\otimes e^{-st\Omega}}ds\cdot tdt D_n E\\
&\qquad  -\frac{T}{2} \int_{0}^{1}\int_{0}^{1}\frac{1}{2}E^{\intercal}D_n^{\intercal}K_{n,n} \del[1]{e^{-(1-s)t\Omega}Ae^{-(1-t)\Omega}\otimes e^{-st\Omega}+e^{-st\Omega}\otimes e^{-(1-t)\Omega}A e^{-(1-s)t\Omega} }ds\cdot tdt D_n E\\
&=  -\frac{T}{4} E^{\intercal}D_n^{\intercal}\int_{0}^{1}\int_{0}^{1}\del[1]{ e^{-st\Omega}\otimes e^{-(1-s)t\Omega}A e^{-(1-t)\Omega}+e^{-(1-t)\Omega}Ae^{-(1-s)t\Omega}\otimes e^{-st\Omega}}ds\cdot tdt D_n E\\
&\qquad  -\frac{T}{4}E^{\intercal}D_n^{\intercal} \int_{0}^{1}\int_{0}^{1} \del[1]{e^{-(1-s)t\Omega}Ae^{-(1-t)\Omega}\otimes e^{-st\Omega}+e^{-st\Omega}\otimes e^{-(1-t)\Omega}A e^{-(1-s)t\Omega} }ds\cdot tdt D_n E
\end{align*}
where the second equality is due to that $K_{n,n}D_n=D_n$ and symmetry of $K_{n,n}$ (see (52) of \cite{magnusneudecker1986}).  

\bigskip

For part (iii), note that $\mathbb{E}[A]=\mathbb{E}[D^{-1/2}\tilde{\Sigma}_T D^{-1/2}]=\Theta=e^{\Omega}$. Then by merging terms, we have
\begin{align*}
\Upsilon_D& =   \frac{1}{2} E^{\intercal}D_n^{\intercal}\int_{0}^{1}\int_{0}^{1}\del[1]{ e^{-st\Omega}\otimes  e^{st\Omega}+ e^{st\Omega}\otimes e^{-st\Omega}}ds\cdot tdt D_n E.
\end{align*}
To prove the equivalence between (\ref{align Hessian form 1}) and (\ref{align Hessian form 2}), it suffices to show
\begin{equation}
\label{eqn Hessian two forms equal}
\int_{0}^{1}\int_{0}^{1}\del[1]{ e^{-st\Omega}\otimes  e^{st\Omega}+ e^{st\Omega}\otimes e^{-st\Omega}}ds\cdot tdt= \int_{0}^{1}\int_{0}^{1}e^{(t+s-1)\Omega}\otimes e^{(1-t-s)\Omega}dsdt.
\end{equation}

Suppose $\Theta=e^{\Omega}= Q^{\intercal}\text{diag}(\lambda_1,\ldots,\lambda_n)Q$ (orthogonal diagonalization). The eigenvalues $\lambda_j$s are all positive but need not be distinct.  We first consider the first term of (\ref{eqn Hessian two forms equal}). By definition of matrix function, we have
\[e^{-st\Omega}=Q^{\intercal}\text{diag}( \lambda_1^{-st},\ldots, \lambda_n^{-st})Q\qquad e^{st\Omega}=Q^{\intercal}\text{diag}( \lambda_1^{st},\ldots, \lambda_n^{st})Q\]
\begin{align*}
& e^{-st\Omega}\otimes  e^{st\Omega}+e^{st\Omega}\otimes e^{-st\Omega}=\\
&(Q\otimes Q)^{\intercal}\sbr[2]{\text{diag}( \lambda_1^{-st},\ldots, \lambda_n^{-st})\otimes \text{diag}( \lambda_1^{st},\ldots, \lambda_n^{st})+\text{diag}( \lambda_1^{st},\ldots, \lambda_n^{st})\otimes \text{diag}( \lambda_1^{-st},\ldots, \lambda_n^{-st}) }(Q\otimes Q)\\
&=: (Q\otimes Q)^{\intercal}M_1(Q\otimes Q),
\end{align*}
where $M_1$ is an $n^2\times n^2$ diagonal matrix whose $[(i-1)n+j]$th diagonal entry is $\del[1]{\frac{\lambda_j}{\lambda_i}}^{st}+\del[1]{\frac{\lambda_i}{\lambda_j}}^{st}$ for $i,j=1,\ldots,n$. Thus
\begin{align*}
\int_{0}^{1}\int_{0}^{1}\del[1]{ e^{-st\Omega}\otimes  e^{st\Omega}+ e^{st\Omega}\otimes e^{-st\Omega}}ds\cdot tdt= (Q\otimes Q)^{\intercal}\int_{0}^{1}\int_{0}^{1}M_1 ds\cdot t dt (Q\otimes Q),
\end{align*}
where $\int_{0}^{1}\int_{0}^{1}M_1t ds dt$ is an $n^2\times n^2$ diagonal matrix whose  $[(i-1)n+j]$th diagonal entry is 
\[\left\lbrace \begin{array}{cc}
1 & \text{ if } i=j\\
1 & \text{ if } i\neq j, \lambda_i=\lambda_j\\
\frac{1}{\sbr[1]{\log \del[1]{\frac{\lambda_i}{\lambda_j}}}^2}\sbr[2]{\frac{\lambda_i}{\lambda_j}+\frac{\lambda_j}{\lambda_i}-2} & \text{ if } i\neq j, \lambda_i\neq \lambda_j
\end{array} \right. \]
for $i,j=1,\ldots,n$. To see this,
\begin{align*}
&\int_{0}^{1}\int_{0}^{1}\del[2]{\frac{\lambda_j}{\lambda_i}}^{st} t ds dt = \int_{0}^{1} \sbr[4]{\frac{\del[2]{\frac{\lambda_j}{\lambda_i}}^{st}}{\log \del[1]{\frac{\lambda_j}{\lambda_i}}^t}}^1_0 t dt = \frac{1}{\log \del[1]{\frac{\lambda_j}{\lambda_i}}}\int_{0}^{1} \sbr[3]{\del[2]{\frac{\lambda_j}{\lambda_i}}^{t}-1}  dt\\
&=\frac{1}{\sbr[1]{\log \del[1]{\frac{\lambda_j}{\lambda_i}}}^2}\del [3]{\frac{\lambda_j}{\lambda_i}-1-\log \del[3]{\frac{\lambda_j}{\lambda_i}}}.
\end{align*}
Similarly
\[\int_{0}^{1}\int_{0}^{1}\del[2]{\frac{\lambda_i}{\lambda_j}}^{st} t ds dt=\frac{1}{\sbr[1]{\log \del[1]{\frac{\lambda_i}{\lambda_j}}}^2}\del [3]{\frac{\lambda_i}{\lambda_j}-1-\log \del[3]{\frac{\lambda_i}{\lambda_j}}},\]
whence we have
\[\int_{0}^{1}\int_{0}^{1}\sbr[3]{\del[2]{\frac{\lambda_j}{\lambda_i}}^{st}+\del[2]{\frac{\lambda_i}{\lambda_j}}^{st}} t ds dt=\frac{1}{\sbr[1]{\log \del[1]{\frac{\lambda_i}{\lambda_j}}}^2}\sbr[2]{\frac{\lambda_i}{\lambda_j}+\frac{\lambda_j}{\lambda_i}-2}.\]

We now consider the second term of (\ref{eqn Hessian two forms equal}).  By definition of matrix function, we have
\begin{align*}
& e^{(t+s-1)\Omega}  = Q^{\intercal}\text{diag}( \lambda_1^{(t+s-1)},\ldots, \lambda_n^{(t+s-1)})Q\qquad e^{(1-t-s)\Omega} = Q^{\intercal}\text{diag}( \lambda_1^{(1-t-s)},\ldots, \lambda_n^{(1-t-s)})Q\\
& e^{(t+s-1)\Omega}\otimes e^{(1-t-s)\Omega} = (Q\otimes Q)^{\intercal}\sbr[2]{\text{diag}( \lambda_1^{(t+s-1)},\ldots, \lambda_n^{(t+s-1)})\otimes \text{diag}( \lambda_1^{(1-t-s)},\ldots, \lambda_n^{(1-t-s)}) }(Q\otimes Q)\\
&=: (Q\otimes Q)^{\intercal}M_2(Q\otimes Q),
\end{align*}
where $M_2$ is an $n^2\times n^2$ diagonal matrix whose $[(i-1)n+j]$th diagonal entry is $\del[1]{\frac{\lambda_i}{\lambda_j}}^{s+t-1}$ for $i,j=1,\ldots,n$. Thus
\begin{align*}
 \int_{0}^{1}\int_{0}^{1}e^{(t+s-1)\Omega}\otimes e^{(1-t-s)\Omega}dsdt=(Q\otimes Q)^{\intercal}\int_{0}^{1}\int_{0}^{1}M_2dsdt(Q\otimes Q)
\end{align*}
where $\int_{0}^{1}\int_{0}^{1}M_2 ds dt$ is an $n^2\times n^2$ diagonal matrix whose  $[(i-1)n+j]$th diagonal entry is
\[\left\lbrace \begin{array}{cc}
1 & \text{ if } i=j\\
1 & \text{ if } i\neq j, \lambda_i=\lambda_j\\
\frac{1}{\sbr[1]{\log \del[1]{\frac{\lambda_i}{\lambda_j}}}^2}\sbr[2]{\frac{\lambda_i}{\lambda_j}+\frac{\lambda_j}{\lambda_i}-2} & \text{ if } i\neq j, \lambda_i\neq \lambda_j
\end{array} \right. \]
for $i,j=1,\ldots,n$. To see this,
\begin{align*}
&\int_{0}^{1}\int_{0}^{1}\del[2]{\frac{\lambda_i}{\lambda_j}}^{s+t-1} ds dt =\frac{\lambda_j}{\lambda_i}\int_{0}^{1}\del[2]{\frac{\lambda_i}{\lambda_j}}^{s} ds\int_{0}^{1}\del[2]{\frac{\lambda_i}{\lambda_j}}^{t} dt\\
&=\frac{\lambda_j}{\lambda_i}\sbr[3]{\int_{0}^{1}\del[2]{\frac{\lambda_i}{\lambda_j}}^{s} ds}^2=\frac{\lambda_j}{\lambda_i}\sbr[4]{\sbr[4]{\frac{\del [2]{\frac{\lambda_i}{\lambda_j}}^s}{\log \del [2]{\frac{\lambda_i}{\lambda_j}}}}_0^1}^2=\frac{1}{\sbr[1]{\log \del[1]{\frac{\lambda_i}{\lambda_j}}}^2}\frac{\lambda_j}{\lambda_i}\sbr[3]{\frac{\lambda_i}{\lambda_j}-1}^2. 
\end{align*}

Comparing $\int_{0}^{1}\int_{0}^{1}M_1t ds dt$ with $\int_{0}^{1}\int_{0}^{1}M_2 ds dt$, we realise (\ref{eqn Hessian two forms equal}) hold.

\bigskip

For part (iv), using the expression for $\frac{\partial\ell_{T,D}(\theta,\mu)}{\partial\theta^{\intercal}}$ and the fact that it has zero expectation, we have
\begin{align*}
&\mathbb{E}\sbr[3]{\frac{1}{T}\frac{\partial\ell_{T,D}(\theta,\mu)}{\partial\theta^{\intercal}}\frac{\partial\ell_{T,D}(\theta,\mu)}{\partial\theta} }=\frac{T}{4}E^{\intercal}D_n^{\intercal}\Psi \text{var}\del[2]{\ve \del[1]{e^{-\Omega}D^{-1/2}\tilde{\Sigma}_TD^{-1/2}e^{-\Omega}}}\Psi D_nE\\
& =\frac{T}{4}E^{\intercal}D_n^{\intercal}\Psi \del[1]{e^{-\Omega}\otimes e^{-\Omega}}(D^{-1/2}\otimes D^{-1/2}) \text{var}\del[3]{\ve \sbr[3]{\frac{1}{T}\sum_{t=1}^{T}(y_t-\mu)(y_t-\mu)^{\intercal}}}\\
&\qquad \cdot (D^{-1/2}\otimes D^{-1/2})\del[1]{e^{-\Omega}\otimes e^{-\Omega}}\Psi D_nE\\
& =\frac{1}{4}E^{\intercal}D_n^{\intercal}\Psi \del[1]{e^{-\Omega}\otimes e^{-\Omega}}(D^{-1/2}\otimes D^{-1/2}) \text{var}\del[2]{\ve \sbr[1]{(y_t-\mu)(y_t-\mu)^{\intercal}}}\\
&\qquad \cdot (D^{-1/2}\otimes D^{-1/2})\del[1]{e^{-\Omega}\otimes e^{-\Omega}}\Psi D_nE\\
& =\frac{1}{4}E^{\intercal}D_n^{\intercal}\Psi \del[1]{e^{-\Omega}\otimes e^{-\Omega}}(D^{-1/2}\otimes D^{-1/2}) 2D_{n}D_{n}^{+}(\Sigma\otimes\Sigma)(D^{-1/2}\otimes D^{-1/2})\del[1]{e^{-\Omega}\otimes e^{-\Omega}}\Psi D_nE\\
& =\frac{1}{4}E^{\intercal}D_n^{\intercal}\Psi \del[1]{e^{-\Omega}\otimes e^{-\Omega}}(D^{-1/2}\otimes D^{-1/2}) (I_{n^2}+K_{n,n})(\Sigma\otimes\Sigma)(D^{-1/2}\otimes D^{-1/2})\del[1]{e^{-\Omega}\otimes e^{-\Omega}}\Psi D_nE\\
& =\frac{1}{4}E^{\intercal}D_n^{\intercal}\Psi \del[1]{e^{-\Omega}\otimes e^{-\Omega}}(D^{-1/2}\otimes D^{-1/2}) (\Sigma\otimes\Sigma)(D^{-1/2}\otimes D^{-1/2})\del[1]{e^{-\Omega}\otimes e^{-\Omega}}\Psi D_nE\\
&\qquad +\frac{1}{4}E^{\intercal}D_n^{\intercal}\Psi \del[1]{e^{-\Omega}\otimes e^{-\Omega}}(D^{-1/2}\otimes D^{-1/2}) K_{n,n}(\Sigma\otimes\Sigma)(D^{-1/2}\otimes D^{-1/2})\del[1]{e^{-\Omega}\otimes e^{-\Omega}}\Psi D_nE\\
& =\frac{1}{4}E^{\intercal}D_n^{\intercal}\Psi \del[1]{e^{-\Omega}\otimes e^{-\Omega}}(D^{-1/2}\otimes D^{-1/2}) (\Sigma\otimes\Sigma)(D^{-1/2}\otimes D^{-1/2})\del[1]{e^{-\Omega}\otimes e^{-\Omega}}\Psi D_nE\\
&\qquad +\frac{1}{4}E^{\intercal}D_n^{\intercal}\Psi \del[1]{e^{-\Omega}\otimes e^{-\Omega}}(D^{-1/2}\otimes D^{-1/2}) (\Sigma\otimes\Sigma)(D^{-1/2}\otimes D^{-1/2})\del[1]{e^{-\Omega}\otimes e^{-\Omega}}K_{n,n}\Psi D_nE\\
& =\frac{1}{2}E^{\intercal}D_n^{\intercal}\Psi \del[1]{e^{-\Omega}\otimes e^{-\Omega}}(D^{-1/2}\otimes D^{-1/2}) (\Sigma\otimes\Sigma)(D^{-1/2}\otimes D^{-1/2})\del[1]{e^{-\Omega}\otimes e^{-\Omega}}\Psi D_nE\\
& =\frac{1}{2}E^{\intercal}D_n^{\intercal}\Psi \del[1]{e^{-\Omega}\otimes e^{-\Omega}}\Psi D_nE,
\end{align*}
where the third equality is due to weak stationarity of $y_t$ and (\ref{eqn U mds}) via Assumption \ref{assu mds}, the fifth equality is due to that $2D_{n}D_{n}^{+}=I_{n^2}+K_{n,n}$, the seventh equality is due to that $K_{n,n}(A\otimes B)=(B\otimes A)K_{n,n}$ for arbitrary $n\times n$ matrices $A$ and $B$, and the second last equality is due to 
\[K_{n,n}\Psi=\int_{0}^{1}K_{n,n}\del[1]{e^{t\Omega}\otimes e^{(1-t)\Omega}}dt=\int_{0}^{1} e^{(1-t)\Omega}\otimes e^{t\Omega} dt=\int_{0}^{1} e^{s\Omega}\otimes e^{(1-s)\Omega} dt=\Psi,\]
via change of variable $1-t\mapsto s$.
\end{proof}

\subsection{Proof of Theorem \ref{thm one step estimator asymptotic normality}}
\label{sec A8}

In this subsection, we give a proof for Theorem \ref{thm one step estimator asymptotic normality}. We will first give some preliminary lemmas leading to the proof of this theorem.

\begin{lemma}
\label{thm fandao bound on exponent} 
For arbitrary $n\times n$ complex matrices $A$ and $E$, and for any matrix norm $\|\cdot\|$,
\begin{align*}
\|e^{A+E}-e^{A}\|\leq\|E\|\exp(\|E\|)\exp(\|A\|).
\end{align*}
\end{lemma}

\begin{proof}
See \cite{hornjohnson1991} Corollary 6.2.32 p430.
\end{proof}

\bigskip

 Define
\[\Xi:=\int_0^1\int_0^1 \Theta^{t+s-1}\otimes \Theta^{1-t-s}dtds\qquad \hat{\Xi}_{T,D}:=\int_0^1\int_0^1 \hat{\Theta}_{T,D}^{t+s-1}\otimes \hat{\Theta}_{T,D}^{1-t-s}dtds\]
such that $\Upsilon_D$ and $\hat{\Upsilon}_{T,D}$ could be denoted $\frac{1}{2}E^{\intercal}D_{n}^{\intercal}\Xi D_{n}E$ and $\frac{1}{2}E^{\intercal}D_{n}^{\intercal}\hat{\Xi}_{T,D} D_{n}E$, respectively.

\begin{lemma}
\label{prop middle of Hessian}
Suppose Assumptions \ref{assu subgaussian vector}%
(i), \ref{assu mixing}, \ref{assu n indexed by T}(i) and \ref{assu about D and Dhat}(i) hold with $1/r_1+1/r_2>1$. Then
\begin{enumerate}[(i)]
\item $\Xi$ has minimum eigenvalue bounded away from zero by an absolute constant and maximum eigenvalue bounded from above by an absolute constant.

\item $\hat{\Xi}_{T,D}$ has minimum eigenvalue bounded away from zero by an absolute constant and maximum eigenvalue bounded from above by an absolute constant with probability approaching 1.

\item
\[\|\hat{\Xi}_{T,D}-\Xi\|_{\ell_2}=O_p\del [3]{\sqrt{\frac{n}{T}}}.\]

\item
\[\|\Psi\|_{\ell_2}=\enVert[3]{\int_{0}^{1}e^{t\Omega}\otimes
e^{(1-t)\Omega}dt}_{\ell_2}=O(1).\]
\end{enumerate}
\end{lemma}

\begin{proof}
The proofs for the first two parts are the same, so we only give one for part (i). Under assumptions of this lemma, we can invoke Lemma \ref{prop mini eigenvalue}(i) in Appendix \ref{sec A4} to have eigenvalues of $\Theta$ to be bounded away from zero and from above by absolute positive constants. Let $\lambda_1,\ldots,\lambda_n$ denote these. We have already shown in the proof of Theorem \ref{prop Haihan score functions and second derivatives} in SM \ref{sec A7} that eigenvalues of $\Xi$ are
\[\left\lbrace \begin{array}{cc}
1 & \text{ if } i=j\\
1 & \text{ if } i\neq j, \lambda_i=\lambda_j\\
\frac{1}{\sbr[1]{\log \del[1]{\frac{\lambda_i}{\lambda_j}}}^2}\sbr[2]{\frac{\lambda_i}{\lambda_j}+\frac{\lambda_j}{\lambda_i}-2} & \text{ if } i\neq j, \lambda_i\neq \lambda_j
\end{array} \right. \]
for $i,j=1,\ldots,n$. This concludes the proof.

For part (iii), we have
\begin{align*}
& \enVert[3]{\int_{0}^{1}\int_{0}^{1}\hat{\Theta}_{T,D}^{t+s-1}\otimes \hat{\Theta}_{T,D}^{1-t-s}dtds-\int_0^1\int_0^1 \Theta^{t+s-1}\otimes \Theta^{1-t-s}dtds}_{\ell_2}\\
&\leq \int_{0}^{1}\int_{0}^{1}\enVert[2]{\hat{\Theta}_{T,D}^{t+s-1}\otimes \hat{\Theta}_{T,D}^{1-t-s}-\Theta^{t+s-1}\otimes \Theta^{1-t-s}}_{\ell_2}dtds\\
&=\int_{0}^{1}\int_{0}^{1}\enVert[2]{\hat{\Theta}_{T,D}^{t+s-1}\otimes \hat{\Theta}_{T,D}^{1-t-s}-\hat{\Theta}_{T,D}^{t+s-1}\otimes \Theta^{1-t-s}+\hat{\Theta}_{T,D}^{t+s-1}\otimes \Theta^{1-t-s}-\Theta^{t+s-1}\otimes \Theta^{1-t-s}}_{\ell_2}dtds\\
&=\int_{0}^{1}\int_{0}^{1}\enVert[2]{\hat{\Theta}_{T,D}^{t+s-1}\otimes (\hat{\Theta}_{T,D}^{1-t-s}- \Theta^{1-t-s})+(\hat{\Theta}_{T,D}^{t+s-1}-\Theta^{t+s-1})\otimes \Theta^{1-t-s}}_{\ell_2}dtds\\
&=\int_{0}^{1}\int_{0}^{1}\sbr[2]{\|\hat{\Theta}_{T,D}^{t+s-1}\|_{\ell_2}\|\hat{\Theta}_{T,D}^{1-t-s}- \Theta^{1-t-s}\|_{\ell_2}+\|\hat{\Theta}_{T,D}^{t+s-1}-\Theta^{t+s-1}\|_{\ell_2}\| \Theta^{1-t-s}\|_{\ell_2}}dtds\\
&\leq \max_{t,s\in [0,1]}\sbr[2]{\|\hat{\Theta}_{T,D}^{t+s-1}\|_{\ell_2}\|\hat{\Theta}_{T,D}^{1-t-s}- \Theta^{1-t-s}\|_{\ell_2}+\|\hat{\Theta}_{T,D}^{t+s-1}-\Theta^{t+s-1}\|_{\ell_2}\| \Theta^{1-t-s}\|_{\ell_2}}.
\end{align*}
First, note that for any $t,s\in [0,1]$, $\|\hat{\Theta}_{T,D}^{t+s-1}\|_{\ell_2}$ and $\|\Theta^{1-t-s}\|_{\ell_2}$ are $O_p(1)$ and $O(1)$, respectively. For example, diagonalize $\Theta$, apply the function $f(x)=x^{1-t-s}$, and take the spectral norm.

The result would then follow if we show that
\[\max_{t,s\in [0,1]}\|\hat{\Theta}_{T,D}^{1-t-s}- \Theta^{1-t-s}\|_{\ell_2}=O_p(\sqrt{n/T}),\quad \max_{t,s\in [0,1]}\|\hat{\Theta}_{T,D}^{t+s-1}-\Theta^{t+s-1}\|_{\ell_2}=O_p(\sqrt{n/T}).\]
It suffices to give a proof for the first equation, as the proof for the second is similar.
\begin{align*}
&\|\hat{\Theta}_{T,D}^{1-t-s}-\Theta^{1-t-s}\|_{\ell_2}=\enVert[1]{e^{(1-t-s)\log \hat{\Theta}_{T,D}}-e^{(1-t-s)\log \Theta}}_{\ell_2}\\
& \leq \|(1-t-s)(\log \hat{\Theta}_{T,D}-\log \Theta)\|_{\ell_2}\exp[(1-t-s)\|\log \hat{\Theta}_{T,D}-\log \Theta\|_{\ell_2}]\exp [(1-t-s)\|\log \Theta\|_{\ell_2}]\\
& = \|(1-t-s)(\log \hat{\Theta}_{T,D}-\log \Theta)\|_{\ell_2}\exp[(1-t-s)\|\log \hat{\Theta}_{T,D}-\log \Theta\|_{\ell_2}]O(1),
\end{align*}
where the first inequality is due to Lemma \ref{thm fandao bound on exponent}, and the second equality is due to the fact that all the eigenvalues of $\Theta$ are bounded away from zero and infinity by absolute positive constants. Now use Theorem \ref{thm main rate of convergence} to get $\|\log \hat{\Theta}_{T,D}-\log \Theta\|_{\ell_2}=O_p\del [1]{\sqrt{\frac{n}{T}}}$. The result follows after recognising $\exp(o_p(1))=O_p(1)$.

The proof for part (iv) is very similar to the one which we gave in the proof of Theorem \ref{prop Haihan score functions and second derivatives} in SM \ref{sec A7}. Since $\Theta= Q^{\intercal}\text{diag}(\lambda_1,\ldots,\lambda_n)Q$, we have $\Theta^t= Q^{\intercal}\text{diag}(\lambda_1^t,\ldots,\lambda_n^t)Q$ and $\Theta^{1-t}= Q^{\intercal}\text{diag}(\lambda_1^{1-t},\ldots,\lambda_n^{1-t})Q$.
Then
\begin{align*}
\Theta^t\otimes \Theta^{1-t}&=(Q\otimes Q)^{\intercal} \sbr[1]{\text{diag}(\lambda_1^t,\ldots,\lambda_n^t)\otimes \text{diag}(\lambda_1^{1-t},\ldots,\lambda_n^{1-t})}(Q\otimes Q)=: (Q\otimes Q)^{\intercal} M_3 (Q\otimes Q),
\end{align*}
where $M_3$ is an $n^2\times n^2$ diagonal matrix whose $[(i-1)n+j]$th diagonal entry is $\lambda_j\del[1]{\frac{\lambda_i}{\lambda_j}}^{t}$ for $i,j=1,\ldots,n$. Thus
\begin{align*}
\Psi=\int_0^1 \Theta^{t}\otimes \Theta^{1-t}dt=(Q\otimes Q)^{\intercal}\int_{0}^{1}M_3dt(Q\otimes Q)
\end{align*}
where $\int_{0}^{1}M_3 dt$ is an $n^2\times n^2$ diagonal matrix whose $[(i-1)n+j]$th diagonal entry is
\[\left\lbrace \begin{array}{cc}
\lambda_i & \text{if }i=j\\
\lambda_i & \text{if }i\neq j, \lambda_i=\lambda_j\\
\frac{\lambda_i-\lambda_j}{\log \lambda_i-\log \lambda_j} & \text{if }i\neq j,  \lambda_i\neq \lambda_j
\end{array}\right. \]
for $i,j=1,\ldots,n$. To see this,
\begin{align*}
&\lambda_j \int_{0}^{1}\del[2]{\frac{\lambda_i}{\lambda_j}}^{t} dt =\lambda_j\sbr[4]{\frac{\del [2]{\frac{\lambda_i}{\lambda_j}}^t}{\log \del [2]{\frac{\lambda_i}{\lambda_j}}}}_0^1 =\frac{1}{\log\del [1]{ \frac{\lambda_i}{\lambda_j}}}\lambda_j\sbr[3]{\frac{\lambda_i}{\lambda_j}-1}. 
\end{align*}
\end{proof}

\bigskip

\begin{lemma}
\label{prop Upsilon}
Suppose Assumptions \ref{assu subgaussian vector}%
(i), \ref{assu mixing}, \ref{assu n indexed by T}(i) and \ref{assu about D and Dhat} hold with $1/r_1+1/r_2>1$. Then
\begin{enumerate}[(i)]
\item
\[\|\hat{\Upsilon}_{T,D}-\Upsilon_D\|_{\ell_2}=O_p\del [3]{sn\sqrt{\frac{n}{T}}}.\]

\item
\[\|\hat{\Upsilon}_{T,D}^{-1}-\Upsilon_D^{-1}\|_{\ell_2}=O_p\del [3]{\varpi^2s\sqrt{\frac{1}{nT}}}.\]

\end{enumerate}
\end{lemma}

\begin{proof}
For part (i),
\begin{align*}
& \|\hat{\Upsilon}_{T,D}-\Upsilon_D\|_{\ell_2}=\frac{1}{2}\|E^{\intercal}D_{n}^{\intercal}(\hat{\Xi}_{T,D}-\Xi)D_{n}E\|_{\ell_2}\leq \frac{1}{2}\|E^{\intercal}\|_{\ell_2}\|D_{n}^{\intercal}\|_{\ell_2}\|\hat{\Xi}_{T,D}-\Xi\|_{\ell_2}\|D_{n}\|_{\ell_2}\|E\|_{\ell_2}\\
&=O(1) \|\hat{\Xi}_{T,D}-\Xi\|_{\ell_2}\|E\|^2_{\ell_2}=O_p\del [3]{sn\sqrt{\frac{n}{T}}},
\end{align*}
where the second equality is due to (\ref{eqn spectral norm for Dnplus and Dn}), and the last equality is due to (\ref{eqn E l2 rate}) and Lemma \ref{prop middle of Hessian}(iii).

For part (ii),
\begin{align*}
&\|\hat{\Upsilon}_{T,D}^{-1}-\Upsilon_D^{-1}\|_{\ell_2}=\|\hat{\Upsilon}_{T,D}^{-1}(\Upsilon_D-\hat{\Upsilon}_{T,D})\Upsilon_D^{-1}\|_{\ell_2}\leq \|\hat{\Upsilon}_{T,D}^{-1}\|_{\ell_2}\|\Upsilon_D-\hat{\Upsilon}_{T,D}\|_{\ell_2}\|\Upsilon_D^{-1}\|_{\ell_2}\\
&=O_p(\varpi^2/n^2)O_p\del [3]{sn\sqrt{\frac{n}{T}}}=O_p\del [3]{s\varpi^2\sqrt{\frac{1}{nT}}},
\end{align*}
where the second last equality is due to (\ref{eqn mineval minus Upsilon}).
\end{proof}

\bigskip

We are now ready to give a proof for Theorem \ref{thm one step estimator asymptotic normality}.

\begin{proof}[Proof of Theorem \ref{thm one step estimator asymptotic normality}]
We first show that $\hat{\Upsilon}_{T,D}$ is invertible with probability approaching 1, so that our estimator $\tilde{\theta}_{T,D}:=\hat{\theta}_{T,D}-\hat{\Upsilon}_{T,D}^{-1}\frac{\partial\ell_{T,D}(\hat{\theta}_{T,D},\bar{y})}{\partial\theta^{\intercal}}/T$ is well defined. It suffices to show that $\hat{\Upsilon}_{T,D}$ has minimum eigenvalue bounded away from zero by an absolute constant with probability approaching one. 
\begin{align*}
& \text{mineval}(\hat{\Upsilon}_{T,D})=\frac{1}{2}\text{mineval}( E^{\intercal}D_{n}^{\intercal}\hat{\Xi}_{T,D}D_{n}E) \geq \text{mineval}(\hat{\Xi}_{T,D})\text{mineval}(D_n^{\intercal}D_n)\text{mineval}(E^{\intercal}E)/2\\
&\geq C\frac{n}{\varpi},
\end{align*}
for some absolute positive constant $C$ with probability approaching one, where the second inequality is due to Lemma \ref{prop middle of Hessian}(ii), Assumption \ref{assu about D and Dhat}(ii), and that $D_n^{\intercal}D_n$ is a diagonal matrix with diagonal entries either 1 or 2. Hence $\hat{\Upsilon}_{T,D}$ has minimum eigenvalue bounded away from zero by an absolute constant with probability approaching one. Also as a by-product
\begin{equation}
\label{eqn mineval minus Upsilon}
\|\hat{\Upsilon}_{T,D}^{-1}\|_{\ell_2}=\frac{1}{\text{mineval}(\hat{\Upsilon}_{T,D})}=O_p\del [3]{\frac{\varpi}{n}}\qquad \|\Upsilon_{D}^{-1}\|_{\ell_2}=\frac{1}{\text{mineval}(\Upsilon_{D})}=O\del [3]{\frac{\varpi}{n}}.
\end{equation}
From the definition of $\tilde{\theta}_{T,D}$, for any $b\in \mathbb{R}^{s}$ with $\|b\|_2=1$ we can write
\begin{align*}
&\sqrt{T}b^{\intercal}\hat{\Upsilon}_{T,D}(\tilde{\theta}_{T,D}-\theta)=\sqrt{T}b^{\intercal}\hat{\Upsilon}_{T,D}(\hat{\theta}_{T,D}-\theta)-\sqrt{T}b^{\intercal}\frac{1}{T}\frac{\partial \ell_{T,D}(\hat{\theta}_{T,D},\bar{y})}{\partial \theta^{\intercal}}\\
&=\sqrt{T}b^{\intercal}\hat{\Upsilon}_{T,D}(\hat{\theta}_{T,D}-\theta)-\sqrt{T}b^{\intercal}\frac{1}{T}\frac{\partial \ell_{T,D}(\theta,\bar{y})}{\partial \theta^{\intercal}}-\sqrt{T}b^{\intercal}\Upsilon_D(\hat{\theta}_{T,D}-\theta)+o_p(1)\\
&=\sqrt{T}b^{\intercal}(\hat{\Upsilon}_{T,D}-\Upsilon_D)(\hat{\theta}_{T,D}-\theta)-b^{\intercal}\sqrt{T}\frac{1}{T}\frac{\partial \ell_{T,D}(\theta,\bar{y})}{\partial \theta^{\intercal}}+o_p(1)
\end{align*}
where the second equality is due to Assumption \ref{assu uniform unverifiable condition} and the fact that $\hat{\theta}_{T,D}$ is $\sqrt{n\varpi\kappa(W)/T}$-consistent. Defining $a^{\intercal}:=b^{\intercal}\hat{\Upsilon}_{T,D}$, we write
\begin{align*}
\sqrt{T}\frac{a^{\intercal}}{\|a\|_2}(\tilde{\theta}_{T,D}-\theta)&=\sqrt{T}\frac{a^{\intercal}}{\|a\|_2}\hat{\Upsilon}_{T,D}^{-1}(\hat{\Upsilon}_{T,D}-\Upsilon_D)(\hat{\theta}_{T,D}-\theta) -\frac{a^{\intercal}}{\|a\|_2}\hat{\Upsilon}_{T,D}^{-1}\sqrt{T}\frac{1}{T}\frac{\partial \ell_{T,D}(\theta,\bar{y})}{\partial \theta^{\intercal}}+\frac{o_p(1)}{\|a\|_2}.
\end{align*}
By recognising that $\|a^{\intercal}\|_2=\|b^{\intercal}\hat{\Upsilon}_{T,D}\|_2\geq \text{mineval}(\hat{\Upsilon}_{T,D})$, we have $\frac{1}{\|a\|_2}=O_p\del [1]{\frac{\varpi}{n}}$. Thus without loss of generality, we have, for any $c\in \mathbb{R}^{s}$ with $\|c\|_2=1$,
\[\sqrt{T}c^{\intercal}(\tilde{\theta}_{T,D}-\theta)=\sqrt{T}c^{\intercal}\hat{\Upsilon}_{T,D}^{-1}(\hat{\Upsilon}_{T,D}-\Upsilon_D)(\hat{\theta}_{T,D}-\theta)-c^{\intercal}\hat{\Upsilon}_{T,D}^{-1}\sqrt{T}\frac{1}{T}\frac{\partial \ell_{T,D}(\theta,\bar{y})}{\partial \theta^{\intercal}}+o_p(\varpi/n).\]
We now determine a rate for the first term on the right side in the preceding display. This is straightforward
\begin{align*}
&\sqrt{T}|c^{\intercal}\hat{\Upsilon}_{T,D}^{-1}(\hat{\Upsilon}_{T,D}-\Upsilon_D)(\hat{\theta}_{T,D}-\theta)|\leq \sqrt{T}\|c\|_2\|\hat{\Upsilon}_{T,D}^{-1}\|_{\ell_2}\|\hat{\Upsilon}_{T,D}-\Upsilon_D\|_{\ell_2}\|\hat{\theta}_{T,D}-\theta\|_2\\
& =\sqrt{T}O_p(\varpi/n)sn O_p(\sqrt{n/T})O_p(\sqrt{n\varpi \kappa(W)/T})=O_p\del[3]{\sqrt{\frac{ n^2 \log ^2 n \varpi^3\kappa(W)}{T}}},
\end{align*}
where the first equality is due to (\ref{eqn mineval minus Upsilon}), Lemma \ref{prop Upsilon}(i) and the rate of convergence for the minimum distance estimator $\hat{\theta}_{T}$ ($\hat{\theta}_{T,D}$). Thus
\begin{align*}
\sqrt{T}c^{\intercal}(\tilde{\theta}_{T,D}-\theta)=-c^{\intercal}\hat{\Upsilon}_{T,D}^{-1}\sqrt{T}\frac{1}{T}\frac{\partial\ell_{T,D}(\theta,\bar{y})}{\partial\theta^{\intercal}}+\text{rem},\quad \text{rem}=O_p\del[3]{\sqrt{\frac{ n^2 \log ^2 n \varpi^3\kappa(W)}{T}}}+o_p(\varpi/n)
\end{align*}
whence, if we divide by $\sqrt{c^{\intercal}\hat{\Upsilon}_{T,D}^{-1}c}$, we have
\begin{align*}
\frac{\sqrt{T}c^{\intercal}(\tilde{\theta}_{T,D}-\theta)}{\sqrt{c^{\intercal}\hat{\Upsilon}_{T,D}^{-1}c}}&=\frac{-c^{\intercal}\hat{\Upsilon}_{T,D}^{-1}\sqrt{T}\frac{\partial \ell_{T,D}(\theta,\bar{y})}{\partial \theta^{\intercal}}/T}{\sqrt{c^{\intercal}\hat{\Upsilon}_{T,D}^{-1}c}}+\frac{\text{rem}}{\sqrt{c^{\intercal}\hat{\Upsilon}_{T,D}^{-1}c}}=:\hat{t}_{os,D,1}+t_{os,D,2}.
\end{align*}
Define
\[t_{os,D,1}:=\frac{-c^{\intercal}\Upsilon_D^{-1}\sqrt{T}\frac{\partial \ell_{T,D}(\theta,\mu)}{\partial \theta^{\intercal}}/T}{\sqrt{c^{\intercal}\Upsilon_D^{-1} c}}.\]
To prove Theorem \ref{thm one step estimator asymptotic normality}, it suffices to show $t_{os,D,1}\xrightarrow{d}N(0,1)$, $\hat{t}_{os,D,1}-t_{os,D,1}=o_p(1)$, and $t_{os,D,2}=o_p(1)$.

\subsubsection{$t_{os,D,1}\xrightarrow{d}N(0,1)$}

We now prove that $t_{os,D,1}$ is asymptotically distributed as a standard normal. %It is not difficult to show $\mathbb{E}[-t_{os,D,1}]=0$ and $\text{var}(-t_{os,D,1})=1$ under assumption of normality (Assumption \ref{assu subgaussian vector}(ii)). 
Write
\begin{align*}
& t_{os,D,1}:=\frac{-c^{\intercal}\Upsilon_D^{-1}\sqrt{T}\frac{\partial \ell_{T,D}(\theta,\mu)}{\partial \theta^{\intercal}}/T}{\sqrt{c^{\intercal}\Upsilon_D^{-1} c}}=\\
&\sum_{t=1}^{T}\frac{-\frac{1}{2}c^{\intercal}\Upsilon_D^{-1} E^{\intercal}D_n^{\intercal}\Psi(\Theta^{-1}\otimes \Theta^{-1} )(D^{-1/2}\otimes D^{-1/2})T^{-1/2}\ve \sbr[1]{(y_t-\mu)(y_t-\mu)^{\intercal}-\mathbb{E}(y_t-\mu)(y_t-\mu)^{\intercal}}}{\sqrt{c^{\intercal}\Upsilon_D^{-1} c}}\\
&=:\sum_{t=1}^{T}U_{os,D,T,n,t}.
\end{align*}
The proof is very similar to that of $t_{D,1}\xrightarrow{d}N(0,1)$ in Section \ref{sec asymptotic normality of t1 prime}. It is straightforward to show that $\{U_{os,D,T,n,t}, \mathcal{F}_{T,n,t}\}$ is a martingale difference sequence. We first investigate that at what rate the denominator $\sqrt{c^{\intercal}\Upsilon_D^{-1}c}$ goes to zero.
\begin{align*}
& c^{\intercal}\Upsilon_D^{-1}c=2c^{\intercal}\del [1]{E^{\intercal}D_{n}^{\intercal}\Xi D_{n}E}^{-1}c\geq 2 \text{mineval}\del[2]{\del [1]{E^{\intercal}D_{n}^{\intercal}\Xi D_{n}E}^{-1}} =\frac{2}{\text{maxeval}\del [1]{E^{\intercal}D_{n}^{\intercal}\Xi D_{n}E}}.
\end{align*}
Since,
\begin{align*}
& \text{maxeval}\del [1]{E^{\intercal}D_{n}^{\intercal}\Xi D_{n}E}\leq \text{maxeval}(\Xi) \text{maxeval}(D_n^{\intercal}D_n)\text{maxeval}(E^{\intercal}E)\leq Csn,
\end{align*}
for some positive constant $C$ because of Lemma \ref{prop middle of Hessian}(i), (\ref{eqn maxeval EE}) and that $D_n^{\intercal}D_n$ is a diagonal matrix with diagonal entries either 1 or 2. Thus we have
\begin{equation}
\label{eqn denominator Upsilon}
\frac{1}{\sqrt{c^{\intercal}\Upsilon_D^{-1}c}}=O(\sqrt{sn}).
\end{equation}

We now verify (i) and (ii) of Theorem \ref{thm mcleish clt} in Appendix \ref{sec oldappendixB}. We consider $|U_{os,D,T,n,t}|$ first.
\begin{align*}
&|U_{os,D,T,n,t}|=\\
&\envert[3]{\frac{\frac{1}{2}c^{\intercal}\Upsilon_D^{-1} E^{\intercal}D_n^{\intercal}\Psi(\Theta^{-1}\otimes \Theta^{-1} )(D^{-1/2}\otimes D^{-1/2})T^{-1/2}\ve \sbr[1]{(y_t-\mu)(y_t-\mu)^{\intercal}-\mathbb{E}(y_t-\mu)(y_t-\mu)^{\intercal}}}{\sqrt{c^{\intercal}\Upsilon_D^{-1} c}}}\\
&\leq \frac{\frac{1}{2}T^{-1/2}\enVert[1]{c^{\intercal}\Upsilon_D^{-1} E^{\intercal}D_n^{\intercal}\Psi(\Theta^{-1}\otimes \Theta^{-1} )(D^{-1/2}\otimes D^{-1/2})}_2\enVert[1]{\ve \sbr[1]{(y_t-\mu)(y_t-\mu)^{\intercal}-\mathbb{E}(y_t-\mu)(y_t-\mu)^{\intercal}}}_2}{\sqrt{c^{\intercal}\Upsilon_D^{-1} c}}\\
&= O \del[3]{\sqrt{\frac{s^2\varpi^2}{T}}}\enVert[1]{(y_t-\mu)(y_t-\mu)^{\intercal}-\mathbb{E}(y_t-\mu)(y_t-\mu)^{\intercal}}_F\\
&\leq  O \del[3]{\sqrt{\frac{n^2s^2\varpi^2}{T}}}\enVert[1]{(y_t-\mu)(y_t-\mu)^{\intercal}-\mathbb{E}(y_t-\mu)(y_t-\mu)^{\intercal}}_{\infty},
\end{align*}
where the second equality is due to (\ref{eqn denominator Upsilon}) and that
\begin{align*}
&\enVert[1]{c^{\intercal}\Upsilon_D^{-1} E^{\intercal}D_n^{\intercal}\Psi(\Theta^{-1}\otimes \Theta^{-1} )(D^{-1/2}\otimes D^{-1/2})}_2\\
&\leq \|\Upsilon_D^{-1}\|_{\ell_2} \|E^{\intercal}\|_{\ell_2}\|D_n^{\intercal}\|_{\ell_2}\|\Psi\|_{\ell_2}\|\Theta^{-1}\otimes \Theta^{-1} \|_{\ell_2}\|D^{-1/2}\otimes D^{-1/2}\|_{\ell_2}=O\del[3]{\frac{\varpi}{n}}\sqrt{sn}=O\del[3]{\sqrt{\frac{s\varpi^2}{n}}}
\end{align*}
via (\ref{eqn mineval minus Upsilon}) and (\ref{eqn E l2 rate}). Next, using a similar argument which we explained in detail in Section \ref{sec asymptotic normality of t1 prime}, we have
\begin{align*}
&\enVert[2]{\max_{1\leq t\leq T}|U_{os,D,T,n,t}|}_{\psi_1}\leq \log (1+T)\max_{1\leq t\leq T}\enVert[1]{U_{os,D,T,n,t}}_{\psi_1}\\
&=\log (1+T)O \del[3]{\sqrt{\frac{n^2s^2\varpi^2}{T}}}\max_{1\leq t\leq T}\enVert[2]{\enVert[1]{(y_t-\mu)(y_t-\mu)^{\intercal}-\mathbb{E}(y_t-\mu)(y_t-\mu)^{\intercal}}_{\infty}}_{\psi_1}\\
&=\log (1+T)\log (1+n^2)O \del[3]{\sqrt{\frac{n^2s^2\varpi^2}{T}}}\max_{1\leq t\leq T}\max_{1\leq i,j\leq n}\enVert[1]{(y_{t,i}-\mu_i)(y_{t,j}-\mu_j)}_{\psi_1}\\
&=O \del[3]{\sqrt{\frac{n^2s^2\varpi^2\log^2 (1+T)\log^2 (1+n^2)}{T}}}=o(1)
\end{align*}
where the last equality is due to Assumption \ref{assu n indexed by T}(iii). Since $\|U\|_{L_r}\leq r!\|U\|_{\psi_1}$ for any random variable $U$ (\cite{vandervaartWellner1996}, p95), we conclude that (i) and (ii) of Theorem \ref{thm mcleish clt} in Appendix \ref{sec oldappendixB} are satisfied.  

We now verify condition (iii) of Theorem \ref{thm mcleish clt} in Appendix \ref{sec oldappendixB}. Since we have already shown that $sn c^{\intercal}\Upsilon_D^{-1}c$ is bounded away from zero by an absolute constant, it suffices to show
\[sn \envert[3]{\frac{1}{T}\sum_{t=1}^{T}\del[3]{\frac{1}{2}c^{\intercal}\Upsilon_D^{-1} E^{\intercal}D_n^{\intercal}\Psi(\Theta^{-1}\otimes \Theta^{-1} )(D^{-1/2}\otimes D^{-1/2})u_t}^2-c^{\intercal}\Upsilon_D^{-1}c}=o_p(1),\]
where $u_t:=\ve \sbr[1]{ (y_t-\mu)(y_t-\mu)^{\intercal}-\mathbb{E}(y_t-\mu)(y_t-\mu)^{\intercal}}$. Under Assumptions \ref{assu subgaussian vector}(ii) and \ref{assu mds}, we have already shown in the proof of part (iv) of Theorem \ref{prop Haihan score functions and second derivatives} that
\begin{align*}
& c^{\intercal}\Upsilon_D^{-1}c=c^{\intercal}\Upsilon_D^{-1}\Upsilon_D\Upsilon_D^{-1}c=c^{\intercal}\Upsilon_D^{-1}\del[3]{\frac{1}{2}E^{\intercal}D_n^{\intercal}\Psi(\Theta^{-1}\otimes \Theta^{-1})\Psi D_n E}\Upsilon_D^{-1}c\\
&= \frac{1}{4}c^{\intercal}\Upsilon_D^{-1}E^{\intercal}D_n^{\intercal}\Psi(\Theta^{-1}\otimes \Theta^{-1})(D^{-1/2}\otimes D^{-1/2})V(D^{-1/2}\otimes D^{-1/2})(\Theta^{-1}\otimes \Theta^{-1})\Psi D_n E\Upsilon_D^{-1}c.
\end{align*}
Thus
\begin{align*}
& sn \envert[3]{\frac{1}{T}\sum_{t=1}^{T}\del[3]{\frac{1}{2}c^{\intercal}\Upsilon_D^{-1} E^{\intercal}D_n^{\intercal}\Psi(\Theta^{-1}\otimes \Theta^{-1} )(D^{-1/2}\otimes D^{-1/2})u_t}^2-c^{\intercal}\Upsilon_D^{-1}c}\\
&\leq \frac{1}{4}sn \enVert[3]{\frac{1}{T}\sum_{t=1}^{T}u_tu_t^{\intercal}-V}_{\infty}\enVert[1]{(D^{-1/2}\otimes D^{-1/2})(\Theta^{-1}\otimes \Theta^{-1})\Psi D_n E\Upsilon_D^{-1}c}_1^2\\
&\leq\frac{1}{4} sn^3 \enVert[3]{\frac{1}{T}\sum_{t=1}^{T}u_tu_t^{\intercal}-V}_{\infty}\enVert[1]{(D^{-1/2}\otimes D^{-1/2})(\Theta^{-1}\otimes \Theta^{-1})\Psi D_n E\Upsilon_D^{-1}c}_2^2\\
&\leq \frac{1}{4}sn^3 \enVert[3]{\frac{1}{T}\sum_{t=1}^{T}u_tu_t^{\intercal}-V}_{\infty}\|D^{-1/2}\otimes D^{-1/2}\|_{\ell_2}^2\|\Theta^{-1}\otimes \Theta^{-1}\|_{\ell_2}^2\|\Psi\|_{\ell_2}^2\|D_n\|_{\ell_2}^2\| E\|_{\ell_2}^2\|\Upsilon_D^{-1}\|_{\ell_2}^2\\
&= O_p(sn^3) \sqrt{\frac{\log n}{T}}\cdot sn \cdot \frac{\varpi^2}{n^2}=O_p\del[3]{\sqrt{\frac{n^4\cdot\log n\cdot \varpi^4\cdot \log^4 n}{T}}}=o_p(1)
\end{align*}
where the first equality is due to (\ref{eqn mineval minus Upsilon}), (\ref{eqn E l2 rate}) and the fact that $\enVert[1]{T^{-1}\sum_{t=1}^{T}u_tu_t^{\intercal}-V}_{\infty}=O_p(\sqrt{\frac{\log n}{T}})$, which can be deduced from the proof of Lemma \ref{lemma rate for hatV-V infty} in SM \ref{sec A5}, and the last equality is due to Assumption \ref{assu n indexed by T}(iii).

\subsubsection{$\hat{t}_{os,D,1}-t_{os,D,1}=o_p(1)$}

We now show that $\hat{t}_{os,D,1}-t_{os,D,1}=o_p(1)$. Let $A_{os,D}$ and $\hat{A}_{os,D}$ denote the numerators of $t_{os,D,1}$ and $\hat{t}_{os,D,1}$, respectively.
\[\hat{t}_{os,D,1}-t_{os,D,1}=\frac{\hat{A}_{os,D}}{\sqrt{c^{\intercal}\hat{\Upsilon}_{T,D}^{-1}c}}-\frac{A_{os,D}}{\sqrt{c^{\intercal}\Upsilon_{D}^{-1}c}}=\frac{\sqrt{sn}\hat{A}_{os,D}}{\sqrt{snc^{\intercal}\hat{\Upsilon}_{T,D}^{-1}c}}-\frac{\sqrt{sn}A_{os,D}}{\sqrt{snc^{\intercal}\Upsilon_{D}^{-1}c}}\]
Since we have already shown in (\ref{eqn denominator Upsilon}) that $snc^{\intercal}\Upsilon_{D}^{-1}c$ is bounded away from zero by an absolute constant, it suffices to show the denominators as well as numerators of $\hat{t}_{os,D,1}$ and $t_{os,D,1}$ are asymptotically equivalent.

\subsubsection{Denominators of $\hat{t}_{os,D,1}$ and $t_{os,D,1}$}

We need to show
\[sn|c^{\intercal}(\hat{\Upsilon}^{-1}_{T,D}-\Upsilon^{-1}_{D})c|=o_p(1).\]
This is straightforward.
\begin{align*}
& sn|c^{\intercal}(\hat{\Upsilon}^{-1}_{T,D}-\Upsilon^{-1}_{D})c|\leq sn\|\hat{\Upsilon}^{-1}_{T,D}-\Upsilon^{-1}_{D})\|_{\ell_2}=snO_p\del [3]{s\varpi^2\sqrt{\frac{1}{nT}}}=O_p\del [3]{s^2\varpi^2\sqrt{\frac{n}{T}}}=o_p(1),
\end{align*}
where the last equality is due to Assumption \ref{assu n indexed by T}(iii).

\subsubsection{Numerators of $\hat{t}_{os,D,1}$ and $t_{os,D,1}$}

We now show
\[\sqrt{sn}\envert[3]{c^{\intercal}\hat{\Upsilon}_{T,D}^{-1}\sqrt{T}\frac{\partial \ell_{T,D}(\theta,\bar{y})}{\partial \theta^{\intercal}}/T-c^{\intercal}\Upsilon_D^{-1}\sqrt{T}\frac{\partial \ell_{T,D}(\theta,\mu)}{\partial \theta^{\intercal}}/T}=o_p(1).\]
Using triangular inequality, we have
\begin{align}
& \sqrt{sn}\envert[3]{c^{\intercal}\hat{\Upsilon}_{T,D}^{-1}\sqrt{T}\frac{\partial \ell_{T,D}(\theta,\bar{y})}{\partial \theta^{\intercal}}/T-c^{\intercal}\Upsilon_D^{-1}\sqrt{T}\frac{\partial \ell_{T,D}(\theta,\mu)}{\partial \theta^{\intercal}}/T}\notag \\
&\leq \sqrt{sn}\envert[3]{c^{\intercal}\hat{\Upsilon}_{T,D}^{-1}\sqrt{T}\frac{\partial \ell_{T,D}(\theta,\bar{y})}{\partial \theta^{\intercal}}/T-c^{\intercal}\Upsilon_D^{-1}\sqrt{T}\frac{\partial \ell_{T,D}(\theta,\bar{y})}{\partial \theta^{\intercal}}/T}\notag\\ &\qquad+\sqrt{sn}\envert[3]{c^{\intercal}\Upsilon_{D}^{-1}\sqrt{T}\frac{\partial \ell_{T,D}(\theta,\bar{y})}{\partial \theta^{\intercal}}/T-c^{\intercal}\Upsilon_D^{-1}\sqrt{T}\frac{\partial \ell_{T,D}(\theta,\mu)}{\partial \theta^{\intercal}}/T}\label{align one step numerator D}
\end{align}

We first show that the first term of (\ref{align one step numerator D}) is $o_p(1)$.
\begin{align*}
&\sqrt{sn}\envert[3]{c^{\intercal}(\hat{\Upsilon}_{T,D}^{-1}-\Upsilon_D^{-1})\sqrt{T}\frac{\partial \ell_{T,D}(\theta,\bar{y})}{\partial \theta^{\intercal}}/T}\\
&= \sqrt{sn}\envert[3]{c^{\intercal}(\hat{\Upsilon}_{T,D}^{-1}-\Upsilon_D^{-1})\sqrt{T}\frac{1}{2}E^{\intercal}D_n^{\intercal}\Psi(\Theta^{-1}\otimes \Theta^{-1})(D^{-1/2}\otimes D^{-1/2})\ve (\hat{\Sigma}_T-\Sigma) }\\
&=O(\sqrt{sn})\|\hat{\Upsilon}_{T,D}^{-1}-\Upsilon_D^{-1}\|_{\ell_2}\sqrt{T}\|E^{\intercal}\|_{\ell_2}\|\hat{\Sigma}_T-\Sigma \|_F=O(\sqrt{sn})\varpi^2 s \sqrt{1/(nT)}\sqrt{T}\sqrt{sn}\sqrt{n}\|\hat{\Sigma}_T-\Sigma \|_{\ell_2}\\
&=O(\sqrt{sn})\varpi^2 s \sqrt{1/(nT)}\sqrt{T}\sqrt{sn}\sqrt{n}\sqrt{n/T} =O_p \del [3]{\sqrt{\frac{n^3 s^4 \varpi^4}{T}}}=o_p(1),
\end{align*}
where the last equality is due to Assumption \ref{assu n indexed by T}(iii).

We now show that the second term of (\ref{align one step numerator D}) is $o_p(1)$.
\begin{align*}
&\sqrt{sn}\envert[3]{c^{\intercal}\Upsilon_{D}^{-1}\sqrt{T}\del[3]{\frac{\partial \ell_{T,D}(\theta,\bar{y})}{\partial \theta^{\intercal}}/T-\frac{\partial \ell_{T,D}(\theta,\mu)}{\partial \theta^{\intercal}}/T}}\\
&= \sqrt{sn}\envert[3]{c^{\intercal}\Upsilon_D^{-1}\sqrt{T}\frac{1}{2}E^{\intercal}D_n^{\intercal}\Psi(\Theta^{-1}\otimes \Theta^{-1})(D^{-1/2}\otimes D^{-1/2})\ve (\hat{\Sigma}_T-\tilde{\Sigma}_T) }\\
&=O(\sqrt{sn})\|\Upsilon_D^{-1}\|_{\ell_2}\sqrt{T}\|E\|_{\ell_2}\|\hat{\Sigma}_T-\tilde{\Sigma}_T\|_F=O_p(\sqrt{sn})\frac{\varpi}{n}\sqrt{T}\sqrt{sn}n \frac{\log n}{T}= O_p\del [3]{\sqrt{\frac{\log^4 n \cdot n^2 \varpi^2 }{T}}}=o_p(1),
\end{align*}
where the third last equality is due to (\ref{eqn rate for hatV-V infinity part D}), and the last equality is due to Assumption \ref{assu n indexed by T}(iii).

\subsubsection{$t_{os,D,2}=o_p(1)$}

To prove $t_{os,D,2}=o_p(1)$, it suffices to show that $\sqrt{sn}|\text{rem}|=o_p(1)$. This is delivered by Assumption \ref{assu n indexed by T}(iii).
\end{proof}

\subsection{Proof of Theorem \ref{thm overidentification test fixed dim} and Corollary \ref{coro diagonal asymptotics}}
\label{sec A9}

In this subsection, we give proofs of Theorem \ref{thm overidentification test fixed dim} and Corollary \ref{coro diagonal asymptotics}.

\begin{proof}[Proof of Theorem \ref{thm overidentification test fixed dim}]
We only give a proof for part (i), as that for part (ii) is similar.  Note that under $H_{0}$,
\begin{align*}
\sqrt{T}g_{T,D}(\theta)  &  =\sqrt{T}[\vech (\log \hat{\Theta}_{T,D})-E\theta]=\sqrt{T}[\vech (\log \hat{\Theta}_{T,D})-\vech (\log\Theta)]\\
& =\sqrt{T}D_{n}^{+}\ve (\log \hat{\Theta}_{T,D}-\log\Theta).
\end{align*}
Thus we can adopt the same method as in Theorem \ref{thm asymptotic normality} to establish the asymptotic distribution of $\sqrt{T}g_{T,D}(\theta)$. In fact, it will be much simpler here because we fixed $n$. We should have
\begin{equation}
\label{eqn gT function clt}
\sqrt{T}g_{T,D}(\theta)\xrightarrow{d}N(0,
S),\qquad S:=D_{n}^{+}H(D^{-1/2}\otimes D^{-1/2})V(D^{-1/2}\otimes D^{-1/2})HD_{n}^{+\intercal},
\end{equation}
where $S$ is positive definite given the assumptions of this theorem. The closed-form solution for $\hat{\theta}_T=\hat{\theta}_{T,D}$ has been given in (\ref{eqn thetaTW closed form solution}), but this is not important. We only need that $\hat{\theta}_{T,D}$ sets the first derivative of the objective function to zero:
\begin{equation}
\label{eqn chi first order condition}
E^{\intercal}Wg_{T,D}(\hat{\theta}_{T,D})=0.
\end{equation}
Notice that
\begin{equation}
\label{eqn chi linearisation}
g_{T,D}(\hat{\theta}_{T,D})-g_{T,D}(\theta)=-E(\hat{\theta}_{T,D}-\theta).
\end{equation}
Pre-multiply (\ref{eqn chi linearisation}) by $\frac{\partial g_{T,D}(\hat{\theta}_{T,D})}{\partial \theta^{\intercal}}W=-E^{\intercal}W$ to give
\[-E^{\intercal}W[g_{T,D}(\hat{\theta}_{T,D})-g_{T,D}(\theta)]=E^{\intercal}WE(\hat{\theta}_{T,D}-\theta),\]
whence we obtain
\begin{equation}
\label{eqn chi make subject}
\hat{\theta}_{T,D}-\theta=-(E^{\intercal}WE)^{-1}E^{\intercal}W[g_{T,D}(\hat{\theta}_{T,D})-g_{T,D}(\theta)].
\end{equation}
Substitute (\ref{eqn chi make subject}) into (\ref{eqn chi linearisation})
\begin{align*}
\sqrt{T}g_{T,D}(\hat{\theta}_{T,D})&=\sbr[1]{I_{n(n+1)/2}-E(E^{\intercal}WE)^{-1}E^{\intercal}W}\sqrt{T}g_{T,D}(\theta)+E(E^{\intercal}WE)^{-1}\sqrt{T}E^{\intercal}Wg_{T,D}(\hat{\theta}_{T,D})\\
&=\sbr[1]{I_{n(n+1)/2}-E(E^{\intercal}WE)^{-1}E^{\intercal}W}\sqrt{T}g_{T,D}(\theta),
\end{align*}
where the second equality is due to (\ref{eqn chi first order condition}). Using (\ref{eqn gT function clt}), we have
\begin{align*}
& \sqrt{T}g_{T,D}(\hat{\theta}_{T,D})\xrightarrow{d} N\del [2]{0, \sbr[1]{I_{n(n+1)/2}-E(E^{\intercal}WE)^{-1}E^{\intercal}W}S\sbr[1]{I_{n(n+1)/2}-E(E^{\intercal}WE)^{-1}E^{\intercal}W}^{\intercal}}.
\end{align*}
Now choosing $W=S^{-1}$, we can simplify the asymptotic covariance matrix in the preceding display to
\[S^{1/2}\del [1]{I_{n(n+1)/2}-S^{-1/2}E(E^{\intercal}S^{-1}E)^{-1}E^{\intercal}S^{-1/2}}S^{1/2}.\]
Thus
\[\sqrt{T}\hat{S}^{-1/2}_{T,D}g_{T,D}(\hat{\theta}_{T,D})\xrightarrow{d}N\del [2]{0,I_{n(n+1)/2}-S^{-1/2}E(E^{\intercal}S^{-1}E)^{-1}E^{\intercal}S^{-1/2}},\]
because $\hat{S}_{T,D}$ is a consistent estimate of $S$ given (\ref{eqn spectral norm of H and Hhat}) and Lemma \ref{lemma rate for hatV-V infty}, which hold under the assumptions of this theorem. The asymptotic covariance matrix in the preceding display is idempotent and has rank $n(n+1)/2-s$. Thus, under $H_0$,
\[Tg_{T,D}(\hat{\theta}_{T,D})^{\intercal}\hat{S}^{-1}_{T,D}g_{T,D}(\hat{\theta}_{T,D})\xrightarrow{d}\chi^2_{n(n+1)/2-s}.\]
\end{proof}

\bigskip

To prove Corollary \ref{coro diagonal asymptotics}, we give the following two auxiliary lemmas.

\begin{lemma}
[\cite{vandervaart1998} p27]\label{lemma asy dist chi}
\[
\frac{\chi^{2}_{k}-k}{\sqrt{2k}}\xrightarrow{d}N(0,1),
\]
as $k\to\infty$.
\end{lemma}

\begin{lemma}
[\cite{vandervaart2010} p41]\label{lemma sequential asymptotics} For
$T,n\in\mathbb{N}$ let $X_{T,n}$ be random vectors such that $X_{T,n}%
\xrightarrow{d} X_{n}$ as $T\to\infty$ for every fixed $n$ such that
$X_{n}\xrightarrow{d} X$ as $n\to\infty$. Then there exists a sequence
$n_{T}\to\infty$ such that $X_{T,n_{T}}\xrightarrow{d} X$ as $T\to\infty$.
\end{lemma}

Now we are ready to give a proof for Corollary \ref{coro diagonal asymptotics}.

\begin{proof}[Proof of Corollary \ref{coro diagonal asymptotics}]
We only give a proof for part (i), as that for part (ii) is similar.  From (\ref{eqn chi fix overiden}) and the Slutsky lemma, we have for every fixed $n$ (and hence $v$ and $s$)
\[\frac{Tg_{T,D}(\hat{\theta}_{T,D})^{\intercal}\hat{S}^{-1}_{T,D}g_{T,D}(\hat{\theta}_{T,D})-\sbr[1]{\frac{n(n+1)}{2}-s}}{\sbr[1]{n(n+1)-2s}^{1/2}}\xrightarrow{d}\frac{\chi^2_{n(n+1)/2-s}-\sbr[1]{\frac{n(n+1)}{2}-s}}{\sbr[1]{n(n+1)-2s}^{1/2}},\]
as $T\to \infty$. Then invoke Lemma \ref{lemma asy dist chi}
\[\frac{\chi^2_{n(n+1)/2-s}-\sbr[1]{\frac{n(n+1)}{2}-s}}{\sbr[1]{n(n+1)-2s}^{1/2}}\xrightarrow{d}N(0,1),\]
as $n\to \infty$ under $H_0$. Next invoke Lemma \ref{lemma sequential asymptotics}, there exists a sequence $n=n_T$ such that
\[\frac{Tg_{T,n,D}(\hat{\theta}_{T,n,D})^{\intercal}\hat{S}_{T,n,D}^{-1}g_{T,n,D}(\hat{\theta}_{T,n,D})-\sbr[1]{\frac{n(n+1)}{2}-s}}{\sbr[1]{n(n+1)-2s}^{1/2}}\xrightarrow{d} N(0,1),\qquad \text{ under }H_0\]
as $T\to \infty$.
\end{proof}

\subsection{Miscellaneous Results}
\label{secSM.cor.cramerwold}

This subsection contains miscellaneous results of the article.

\begin{proof}[Proof of Corollary \ref{cor cramerwold}]
Note that Theorem \ref{thm asymptotic normality} and a result we proved before, namely,
\begin{equation}
\label{eqn another version nGhat-nG}
|c^{\intercal}\hat{J}_{T,D}c-c^{\intercal}J_Dc|=o_p\del [3]{\frac{1}{sn\kappa(W)}},
\end{equation}
imply
\begin{equation}
\label{eqn another version of MD clt}
\sqrt{T}c^{\intercal}(\hat{\theta}_{T,D}-\theta^0)\xrightarrow{d}N(0,c^{\intercal}J_Dc).
\end{equation}
Consider an arbitrary, non-zero vector $b\in \mathbb{R}^k$. Then
\[\enVert[3]{\frac{Ab}{\|Ab\|_2}}_2=1,\]
so we can invoke (\ref{eqn another version of MD clt}) with $c=Ab/\|Ab\|_2$:
\[\sqrt{T}\frac{1}{\|Ab\|_2}b^{\intercal}A^{\intercal}(\hat{\theta}_{T,D}-\theta^0)\xrightarrow{d}N\del [3]{0,\frac{b^{\intercal}A^{\intercal}}{\|Ab\|_2}J_D\frac{Ab}{\|Ab\|_2}},\]
which is equivalent to
\[\sqrt{T}b^{\intercal}A^{\intercal}(\hat{\theta}_{T,D}-\theta^0)\xrightarrow{d}N\del [1]{0,b^{\intercal}A^{\intercal}J_DAb}.\]
Since $b\in \mathbb{R}^k$ is non-zero and arbitrary, via the Cramer-Wold device, we have
\[\sqrt{T}A^{\intercal}(\hat{\theta}_{T,D}-\theta^0)\xrightarrow{d}N\del [1]{0,A^{\intercal}J_DA}.\]
Since we have shown in the mathematical display above (\ref{eqn maxeval EE}) that $J_D$ is positive definite and $A$ has full-column rank, $A^{\intercal}J_DA$ is positive definite and its negative square root exists. Hence,
\[\sqrt{T}(A^{\intercal}J_DA)^{-1/2}A^{\intercal}(\hat{\theta}_{T,D}-\theta^0)\xrightarrow{d}N\del [1]{0,I_k}.\]
Next from (\ref{eqn another version nGhat-nG}),
\[\envert[1]{b^{\intercal}Bb}:=\envert[1]{b^{\intercal}A^{\intercal}\hat{J}_{T,D}Ab-b^{\intercal}A^{\intercal}J_DAb}=o_p\del [3]{\frac{1}{sn\kappa(W)}}\|Ab\|_2^2\leq o_p\del [3]{\frac{1}{sn\kappa(W)}}\|A\|_{\ell_2}^2\|b\|_2^2.\]
By choosing $b=e_j$ where $e_j$ is a vector in $\mathbb{R}^k$ with $j$th component being 1 and the rest of components being 0, we have for $j=1,\ldots,k$
\[\envert[1]{B_{jj}}\leq o_p\del [3]{\frac{1}{sn\kappa(W)}}\|A\|_{\ell_2}^2=o_p(1),\]
where the equality is due to $\|A\|_{\ell_2}=O(\sqrt{sn\kappa(W)})$. By choosing $b=e_{ij}$, where $e_{ij}$ is a vector in $\mathbb{R}^k$ with $i$th and $j$th components being $1/\sqrt{2}$ and the rest of components being 0, we have
\begin{align*}
&\envert[1]{B_{ii}/2+B_{jj}/2+B_{ij}}\leq o_p\del [3]{\frac{1}{sn\kappa(W)}}\|A\|_{\ell_2}^2=o_p(1).
\end{align*}
Then
\[|B_{ij}|\leq |B_{ij}+B_{ii}/2+B_{jj}/2|+|-(B_{ii}/2+B_{jj}/2)|=o_p(1).\]
Thus we proved
\[B=A^{\intercal}\hat{J}_{T,D}A-A^{\intercal}J_DA=o_p(1),\]
because the dimension of the matrix $B$, $k$, is finite. By Slutsky's lemma
\[\sqrt{T}(A^{\intercal}\hat{J}_{T,D}A)^{-1/2}A^{\intercal}(\hat{\theta}_{T,D}-\theta^0)\xrightarrow{d}N\del [1]{0,I_k}.\]
\end{proof}

\bigskip

\begin{lemma}
\label{prop H inverse}
For any positive definite matrix $\Theta$,
\[\del [3]{\int_{0}^{1}[t(\Theta-I)+I]^{-1}\otimes\lbrack t(\Theta-I)+I]^{-1}dt}^{-1}=\int_{0}^{1}e^{t\log\Theta}\otimes e^{(1-t)\log\Theta}dt.\]
\end{lemma}

\begin{proof}
(11.9) and (11.10) of \cite{higham2008} p272 give, respectively, that
\[\ve E=\int_{0}^{1}e^{t\log\Theta}\otimes
e^{(1-t)\log\Theta}dt \ve L(\Theta,E),\]
\[\ve L(\Theta,E)=\int_0^1  [t(\Theta-I)+I]^{-1}\otimes\lbrack
t(\Theta-I)+I]^{-1}dt\ve E.\]
Substitute the preceding equation into the second last
\[\ve E=\int_{0}^{1}e^{t\log\Theta}\otimes
e^{(1-t)\log\Theta}dt \int_0^1  [t(\Theta-I)+I]^{-1}\otimes\lbrack
t(\Theta-I)+I]^{-1}dt\ve E.\]
Since $E$ is arbitrary, the result follows.
\end{proof}

\bigskip

\begin{example}
\label{ex GD special case}
In the special case of normality, $V=2D_{n}D_{n}^{+}(\Sigma\otimes\Sigma)$
(\cite{magnusneudecker1986} Lemma 9). Then $c^{\intercal}J_Dc$ could be simplified into
\begin{align*}
&  c^{\intercal}J_Dc=\\
&  2c^{\intercal}(E^{\intercal}WE)^{-1}E^{\intercal}WD_{n}^{+}H(D^{-1/2}%
\otimes D^{-1/2})D_{n}D_{n}^{+}(\Sigma\otimes\Sigma)(D^{-1/2}\otimes
D^{-1/2})HD_{n}^{+^{\intercal}}WE(E^{\intercal}WE)^{-1}c\\
&  =2c^{ \intercal}(E^{\intercal}WE)^{-1}E^{\intercal}WD_{n}^{+}%
H(D^{-1/2}\otimes D^{-1/2})(\Sigma\otimes\Sigma)(D^{-1/2}\otimes D^{-1/2}
)HD_{n}^{+^{\intercal}}WE(E^{\intercal}WE)^{-1}c\\
&  =2c^{\intercal}(E^{\intercal}WE)^{-1}E^{\intercal}WD_{n}^{+}H(D^{-1/2}%
\Sigma D^{-1/2}\otimes D^{-1/2}\Sigma D^{-1/2})HD_{n}^{+^{\intercal}%
}WE(E^{\intercal} WE)^{-1}c\\
&  =2c^{\intercal}(E^{\intercal}WE)^{-1}E^{\intercal}WD_{n}^{+}H(\Theta
\otimes\Theta)HD_{n}^{+^{\intercal}}WE(E^{\intercal}WE)^{-1}c,
\end{align*}
where the second equality is true because, given the structure of $H$, via Lemma
11 of \cite{magnusneudecker1986}, we have the following identity:
\[D_{n}^{+}H(D^{-1/2}\otimes D^{-1/2})=D_{n}^{+}H(D^{-1/2}\otimes D^{-1/2}%
)D_{n}D_{n}^{+}.\]
%Note that Assumption \ref{assu mineval of V} is automatically satisfied under normality given Assumptions \ref{assu n indexed by T}(i) and \ref{assu about D and Dhat}(i) (via Proposition \ref{prop mini eigenvalue} in Appendix A.4).
\end{example}

\bigskip

%\begin{lemma}
%\label{prop product of L subgaussian} 
%Suppose we have subgaussian random
%variables $Z_{l,t,j}$ for $l=1,\ldots,L$ ($L\geq2$ fixed), $t=1,\ldots,T$ and
%$j=1,\ldots,p$. $Z_{l_{1},t_{1},j_{1}}$ and $Z_{l_{2},t_{2},j_{2}}$ are
%independent as long as $t_{1}\neq t_{2}$ regardless of the values of other
%subscripts. Then,
%\[
%\max_{1\leq j\leq p}\max_{1\leq t\leq T}\mathbb{E}%
%\envert[3]{ \prod_{l=1}^LZ_{l,t,j}}\leq A=O(1),
%\]
%for some positive constant $A$ and
%\[
%\max_{1\leq j \leq p}%
%\envert[2]{ \frac{1}{T}\sum_{t=1}^{T}\del[2]{ \prod_{l=1}^LZ_{l,t,j}-\mathbb{E}\sbr[2]{ \prod_{l=1}^LZ_{l,t,j}} } }
%=O_{p}\del[2]{\sqrt{\frac{(\log p)^{L+1}}{T}}}.
%\]
%\end{lemma}

%\begin{proof}
%See Proposition 3 of \cite{kocktang2014}.
%\end{proof}

\bibliographystyle{ecta}
\bibliography{KronHLT_Biblio}

\end{document}